\documentclass{amsart}
\usepackage{amsmath}
\usepackage{amsfonts}
\usepackage{amssymb}
\usepackage{epsfig}
\newtheorem{theorem}{Theorem}[section]

\newtheorem{lm}[theorem]{Lemma}

\newtheorem{cor}[theorem]{Corollary}

\newtheorem{rem}[theorem]{Remark}
\newtheorem{pr}[theorem]{Proposition}
\usepackage{color}

\begin{document}
\title{Some homological properties of $GL(m|n)$ in arbitrary characteristic}
\author{Alexandr N. Zubkov}
\address{Omsk State Pedagogical University, Chair of Mathematical Analysis, Algebra and Geometry, 644099 Omsk-99, Tuhachevskogo Embankment 14, Russia}
\email{a.zubkov@yahoo.com}

\maketitle

\section*{Introduction}

Let $G$ be a reductive group defined over a field $K$ of zero characteristic and let $B$ be a Borel subgroup of $G$. 
The Borel-Bott-Weil theorem describes the cohomology $H^{\bullet}(G/B, K_{\lambda})$. This classical result  
was generalized by Penkov for almost all series of simple Lie supergroups with respect to an additional condition on weights $\lambda$ (cf. \cite{pen}). The Penkov's approach to the proof of super BBW theorem is based on the Demazure's idea to use minimal parabolic subgroups \cite{dem}.  
The aim of the present article is to demonstrate how Demazure-Penkov's approach can be extended for general linear supergroups over a field of positive characteristic. We prove a superanalog of Mackey imprimitivity theorem (cf. \cite{cps}) and derive some 
standard facts about cohomologies $H^{\bullet}(G/H, ?)$ to realize the proof of super BBW theorem  in the way that mimics  \cite{jan}, II.5. Besides, we prove a partial generalization of Kempf's vanishing theorem that can be formulated as follows. Let $G=GL(m|n)$ and $B$ is a Borel supersubgroup of $G$. If a weight $\lambda$ satisfies $(\lambda, \beta^{\vee}_i)\geq k_i$, where $\beta_i$ runs over simple positive roots of $B_{ev}$ and $k_i$ is a certain non-negative integer depending of $\beta_i$, then $H^k(G/B, K_{\lambda})=0$ for all $k\geq 1$. 
This theorem can be directly deduced from \cite{br}, Theorem 2.7, once we show that $G/B$ is a locally decomposable superscheme (the condition $Q5$ in \cite{br}). 

To avoid the proof of this non-trivial statement we develop a different approach. First of all, one can prove the above theorem over a field of positive characteristic using some nice properties of the Frobenius kernels. More precisely, it can be easily shown that a sheaf quotient $GB_r/B$ is an affine decomposable superscheme. By Theorem 2.7, \cite{br}, $ind^{BG_r}_B K^{\epsilon}_{\lambda}|_{B_{ev}G_{ev,r}}$ has a filtration with quotients
that are isomorphic to $ind^{B_{ev}G_{ev, r}}_{B_{ev}} K_{\lambda-\pi}$, where $\pi$ runs over sums of roots $\alpha\in\Phi_1^+$ without repetitions. The isomorphism 
$$H^k(G/B, ?)|_{G_{ev}}\simeq H^k(G_{ev}/B_{ev}G_{ev,r}, (ind^{BG_r}_B ?)|_{B_{ev}G_{ev, r}})$$
and the standard long exact sequence arguments infer our statement.
Next, we observe that if 
$H^k(G/B, K_{\lambda})\neq 0$, where $K$ is a field of zero characteristic, then $H^k(G/B, F_{\lambda})\neq 0$ for any field
$F$. The advantage of our approach is that we do not need the property of local decomposability.

The paper is organized as follows. In the first ten sections we give all necessary definitions, notations and derive auxiliary results. The most important results in these sections are Theorem \ref{simpleimprim} and Corollary \ref{reductiontoeven}.   
In the eleventh section we describe the representations of minimal parabolic supersubgroups that plays crucial role in the proof of super BBW theorem in the next section. In the thirteenth section we obtain the characteristic free character formula of Euler characteristic $\chi(B, \lambda^{\epsilon})$ (cf. \cite{grusserg, grusserg2}). The main result of the next section, a partial generalization of Kempf's vanishing theorem, has been discussed above. The last section is devoted to the complete description of the cohomology $H^{\bullet}(G/B, K^{\epsilon}_{\lambda})$, where $G=GL(2|1)$ and $B$ is a non-standard Borel supersubgroup of $G$.

\section{Supermodules and supercomodules}

A {\it vector superspace} is a vector space graded by the group $\mathbb{Z}_2 = \{0, 1\}$. The homogeneous components of $V$ are denoted by $V_0$, $V_1$. The degree of a homogeneous 
element, say $v$, is denoted by $|v|$. If $V$ and $W$ are superspaces, then $\mathsf{Hom}_K(V, W)$ has the natural superspace structure defined by 
$$\mathsf{Hom}_K(V, W)_i=\{\phi| \phi(V_j)\subseteq W_{i+j}, i, j\in\mathbb{Z}_2\} .$$
We let $\mathsf{SMod}_K$ denote the $K$-linear abelian category of vector superspaces with even morphisms.
This forms a tensor category with the canonical symmetry
$$ t=t_{V, W} :V \otimes W \overset{\simeq}{\longrightarrow} W \otimes V, \quad v \otimes w \mapsto (-1)^{|v||w|} w \otimes v,$$
where $V, W \in \mathsf{SMod}_K$. 

Objects defined in this symmetric tensor category are called with the adjective `super' attached. For example,
a {\it (Hopf) superalgebra} is a (Hopf) algebra object in $\mathsf{SMod}_K$. All superalgebras are assumed to be unital.  

A superalgebra $A$ is called
supercommutative, if $ab=(-1)^{|a||b|}ba$ for all homogeneous elements $a, b\in A$. A typical example of a supercommutative superalgebra is a symmetric superalgebra $S(V)$ of superspace $V$. More precisely, $S(V)=T(V)/I$, where $T(V)=
\oplus_{k\geq 0} V^{\otimes k}$ is a tensor superalgebra and the ideal $I$ is generated by the elements
$v\otimes w-(-1)^{|v||w|}w\otimes v, v, w\in V$.

Let $\mathsf{SAlg}_K$ denote the category of supercommutative superalgebras.
Given $A \in \mathsf{SAlg}_K$, we let ${}_A\mathsf{SMod}$,\ $\mathsf{SMod}_A$ denote the category of left and respectively, 
right $A$-supermodules.
These two categories are identified if we regard each $M \in \mathsf{SMod}_A$ as an object in ${}_A\mathsf{SMod}$ 
by defining the left $A$-action  
$$
am := (-1)^{|a||m|}ma, \quad a \in A,\ m \in M.
$$
on the supervector space $M$. We remark that $M$ thus turns into an $(A, A)$-superbimodule.

Let $C$ be a supercoalgebra with comultiplication $\Delta_C$ and counit $\epsilon_C$.  We let $\mathsf{SMod}^C$, ${}^C\mathsf{SMod}$ 
denote the categories of right and respectively, left $C$-supercomodules (with even morphisms). For a $C$-supercomodule $V$
let $\tau_V$ denote its comodule map. We use Sweedler's notation $\tau_V(v)=\sum v_1\otimes c_2, v, v_1\in V, c_2\in C$ (or symmetrically, $\tau_V(v)=\sum c_2\otimes v_1$).

If $V\in \mathsf{SMod}^C$ and $W\in {}^C\mathsf{SMod}$, then define a \emph{cotensor product} 
$$V\Box_C W=\{x\in V\otimes W | (\tau_V\otimes id_W-id_V\otimes\tau_W)(x)=0\}.$$

Let $(V, \tau_1)\in\mathsf{SMod}^{C_1}$ and $(V, \tau_1)\in\mathsf{SMod}^{C_2}$.
We say that $\tau_1$ commutes with $\tau_2$ whenever $(id_V\otimes t)(\tau_2\otimes id_{C_1})\tau_1 = (\tau_1\otimes id_{C_2})\tau_2$. 

A Hopf superalgebra $A$ has two right $A$-supercomodule structures, say $A_r$ and $A_l$, given by $\rho_r=\Delta_A$ and $\rho_l=t(s_A\otimes id_A)\Delta_A$ respectively. Here $s_A$ is the antipode of $A$. Besides, $s_A$ takes $A_r$ isomorphically to $A_l$. It is also clear that $\rho_r$ commutes with $\rho_l$. 

If $V\in \mathsf{SMod}^A$ (the case $V\in {}^A\mathsf{SMod}$ is symmetric), then $V^*$ has a right $A$-supermodule structure such that
$\tau_{V^*}(\phi)=\sum\phi_1\otimes a_2$ if and only if $$\sum(-1)^{|a_2||v_1|}\phi_1(v_1)a_2 a'_2=\phi(v)$$
for any $v\in V$, where $\tau_V(v)=\sum v_1\otimes a'_2$.
The functor $V\to V^*$ is an anti-equivalence on the full subcategory consisting of all finite dimensional $A$-supercomodules.
\begin{rem}\label{moregeneral}
Replacing the field $K$ by a superalgebra $A\in\mathsf{SAlg}_K$, one can define all the above objects in the tensor symmetric category of $A$-supermodules. In what follows we reserve the notation $\mathsf{SAlg}_A$ for the category of supercommutative
$A$-superalgebras.      
\end{rem}
If a Hopf superalgebra $A$ has a form $B\otimes_{\mathbb{Z}} K$, where $B$ is a Hopf superring (i.e. a Hopf algebra object in $\mathsf{SMod}_{\mathbb{Z}}$), then we say that $B$ is a $\mathbb{Z}$-{\it form} of $A$. An right $A$-supercomodule $V$ has a $\mathbb{Z}$-form $W$, provided $W$ is an right $B$-supercomodule and $V=W\otimes_{\mathbb{Z}} K$. In other words,
$\tau_V(w\otimes a)=\sum (w_1\otimes 1)\otimes (c_2\otimes a)$, where $\tau_W(w)=\sum w_1\otimes c_2, w, w_1\in W, c_2\in B, a\in K$. The case of left supercomodules is symmetric.     
 
\section{$K$-functors}

Any functor from $\mathsf{SAlg}_K$ to the category of sets is called $K$-functor. For example, if
$V$ is a superspace, then one can define a $K$-functor $V_a(A)=V\otimes A, A\in\mathsf{SAlg}_K$.

A $K$-functor $X$ is said to be an {\it affine superscheme}, if $X$ is represented by a superalgebra
$A\in\mathsf{SAlg}_K$. In other words, $X(B)=\mathsf{Hom}_{\mathsf{SAlg}_K}(A, B), B\in\mathsf{SAlg}_K$. 
In notations from \cite{maszub}, $X=SSp \ A$. The category of affine superschemes is anti-equivalent to $\mathsf{SAlg}_K$. For any morphism of affine superschemes $g : SSp \ A\to SSp \ B$ let $g^*$ denote its dual comorphism $g^* : B\to A$.

For example, the functor 
$X(B)=B_0^{m}\oplus B_1^n$ is an affine superscheme represented by the {\it free superalgebra}
$K[x_{i}|1\leq i\leq m+n]$ with $m$ even (free) generators $x_1, \ldots , x_m$ and $n$ odd (free) generators
$x_{m+1}, \ldots, x_{m+n}$. This affine superscheme is called an {\it affine superspace} of (super)dimension 
$m|n$ and it is denoted by $\mathsf{A}^{m|n}$. 

For any $K$-functor $X$ we denote $\mathsf{Mor}(X, \mathsf{A}^{1|1})$ by $K[X]$.
It has a natural superalgebra structure and we call $K[X]$ a {\it coordinate superalgebra} of $X$.
If $X=SSp \ A$, then $K[X]\simeq A$ (cf. \cite{jan}, I.1.3, \cite{maszub}, Lemma 1.1). 
   
A closed subfunctor $Y$ of 
$SSp \ A$ is uniquely defined by a superideal $I_Y =I$ of $A$ such that $Y(B)=\{x\in SSp \ A(B)| x(I)=0\}$. 
Thus $Y=V(I)\simeq SSp \ A/I$ is again affine superscheme (cf. \cite{jan}, I.1.4, \cite{maszub}, \S 3).
We denote the canonical epimorphism $A\to A/I$ by $\pi_Y$.

An open subfunctor $Y$ of $SSp \ A$ is also defined by a superideal $I$ as $Y(B)=\{x\in SSp \ A | Bx(I)=B\}$
(cf. \cite{jan}, I.1.5, \cite{maszub}, \S 3).
In general, $Y=D(I)$ is not isomorphic to any affine superscheme, but if $I=Af, f\in A_0$, then $D(I)\simeq
SSp \ A_f$. Finally, if $X$ is a $K$-functor, then a subfunctor $Y\subseteq X$ is called closed (open), whenever
for any superalgebra $A\in\mathsf{SAlg}_K$ and any its superideal $I$, and for any morphism $\alpha : SSp \ A \to X$,
the subfunctor $\alpha^{-1}(V(I))$ ($\alpha^{-1}(D(I)$) is closed (respectively, open).     

For a $K$-functor $X$ define a subfunctor $X_{ev}(A)=X(\iota^A_{A_0})X(A_0), A\in \mathsf{SAlg}_K$, where $\iota^A_{A_0}$ is the natural algebra embedding $A_0\to A$. For example, $(SSp \ R)_{ev}$ is a closed supesubrscheme of $SSp \ R$, defined by the ideal $RR_1$.  

A local $K$-functor $X$, that has an \emph{open covering} by affine subsuperschemes $X_i\simeq SSp \ R_i , i\in I,$ is called just \emph{superscheme}. 
A superscheme $X$ is called \emph{Noetherian} if and only if the above open covering is finite and each $R_i$ is a Noetherian algebra. If $X$ is a superscheme, then $X_{ev}$ is a closed subfunctor in $X$. In particular, $X_{ev}$ is also superscheme (see \cite{jan, maszub, zub1} for more definitions).
\begin{rem}\label{moregeneral2}
All the above objects can be defined over any supercommutative superalgebra $A$. 
For example, an $A$-functor is a functor from the category $\mathsf{SAlg}_A$ to the category of sets, (affine) $A$-superscheme 
$SSp_A \ B$ is defined by $SSp_A \ B (C)=\mathsf{Hom}_{\mathsf{SAlg}_A}(B, C)$, where $B, C\in\mathsf{SAlg}_A$.
\end{rem}

\section{Supergroups}

An affine superscheme $G=SSp \ A$ is a group $K$-functor if and only if $A$ is a Hopf superalgebra. If it is the case, then $G$ is called an {\it affine supergroup}. Besides, if $A=K[G]$ is finitely generated, then $G$ is called an {\it algebraic supergroup}. 

A closed supersubscheme $H$ of $G$ is a subgroup functor if and only if $I_H$ is a Hopf superideal of $K[G]$ (cf. \cite{jan, zub1, zub2}). In what follows all supersubgroups are supposed to be closed unless otherwise stated. If $H$ is a supersubgroup of $G$, we denote $H\leq G$. For example, $G_{ev}\leq G$ and $G_{ev}$ is called the {\it largest even} supersubgroup of $G$. The affine group $G_{res}=G|_{\mathsf{Alg}_K}$ is isomorphic to $Sp \ K[G]/K[G]K[G]_1$.

If $char K=p>0$ and $K$ is perfect, then we have an $r$-th {\it Frobenius morphism} $F^r : K[G]^{(r)}\to K[G]$ of Hopf superalgebras. Remind that $K[G]^{(r)}$ coincides with  $K[G]$ as an Hopf superring but each $a\in K$ acts as $a^{p^{-r}}$ on $K[G]^{(r)}$ (see \cite{jan, zub1}). Besides, $F^r(f)=f^{p^r}, f\in K[G]^{(r)}$. Denote the dual morphism 
$G\to SSp \ K[G]^{(r)}$ by $f_r$. The normal supersubgroup $G_r=\ker f_r$ is called the $r$-th {\it infinitesimal} supersubgroup. 

\section{Superalgebras of distributions}

Let $X$ be an affine superscheme and $\mathfrak{m}$ be a maximal superideal of $K[X]$. Let $\mathsf{Dist}(X, \mathfrak{m})$ denote
the superspace of distributions with support at $\mathfrak{m}$ (see \cite{jan, zub1} for more details). If $\mathfrak{m}$ is nilpotent, then $\mathsf{Dist}(X, \mathfrak{m})=K[X]^*$.  
For any morphism of affine superschemes
$g : X\to Y$ we denote the induced morphism of superspaces $\mathsf{Dist}(X, \mathfrak{m})\to\mathsf{Dist}(X, (g^*)^{-1}(\mathfrak{m}))$ by $dg_{\mathfrak{m}}$. We call $dg_{\mathfrak{m}}$ a {\it differential} of $g$ at $\mathfrak{m}$. 

Let $G$ be an algebraic supergroup. Then $\mathsf{Dist}(G, \ker\epsilon_G)$ is denoted by $\mathsf{Dist}(G)$. The superspace $\mathsf{Dist}(G)$ is a cocommutative Hopf superalgebra.
For any morphism $f : G\to Y$ the differential $df_{\ker\epsilon_G}$ is denoted by $df$. 

Assume that $H_1$ and $H_2$ are supersubgroups of $G$. We have an morphism of superschemes
$m : H_1\times H_2\to G$ induced by the multiplication of $G$. Then $\mathsf{Dist}(H_1)$ and $\mathsf{Dist}(H_2)$
are Hopf supersubalgebras of $\mathsf{Dist}(G)$ and the morphism of superspaces   
$$dm : \mathsf{Dist}(H_1)\otimes \mathsf{Dist}(H_2)=\mathsf{Dist}(H_1\times H_2, \ker\epsilon_{H_1}\otimes K[H_2] +
K[H_1]\otimes\ker\epsilon_{H_2})\to \mathsf{Dist}(G)$$ is induced by the multiplication of $\mathsf{Dist}(G)$
(cf. \cite{jan}, Part I, 7.4(2)).  

\section{Actions and representations}

Let $X$ be an affine superscheme. Assume that an affine supergroup $G$ acts on $X$ on the right. It is equivalent to the condition that $K[X]\in\mathsf{SMod}^{K[G]}$ and $\tau_{X}=\tau_{K[X]}$ is a superalgebra morphism. 
A left action of $G$ on $X$ is defined symmetrically.

For example, $\rho_r$ and $\rho_l$ are 
corresponding to the right actions $m_r : G\times G\to G$ and $m_l : G\times G\to G$ respectively, where
$m_r(g_1, g_2)=g_1 g_2$ and $m_l(g_1, g_2)=g_2^{-1}g_1 , g_1, g_2\in G(A), A\in \mathsf{SAlg}_K$. Besides, $G$ acts on itself by conjugations, $(g_1, g_2)\mapsto g_2^{-1}g_1g_2$. The corresponding comorphism coincides with $\nu_l(f)=\sum (-1)^{|f_1||f_2|} f_2\otimes s_G(f_1)f_2, f\in K[G]$.
In particular, $H\unlhd G$ if and only if $\nu_l(I_H)\subseteq I_H\otimes K[G]$ (cf. \cite{zub1}).

By definition, the category of left/right $G$-supermodules coincides with the category of right/left $K[G]$-supercomodules. 
Denote them by $G-smod$ and $smod-G$ respectively.
Both categories have an endofunctor $V\to \Pi V$, called {\it parity shift}, such that $\Pi V$ coincides with $V$ as a $K[G]$-comodule and $(\Pi V)_i=V_{i+1}, i=0, 1$, where the sum $i+1$ is computed in $\mathbb{Z}_2$.  

There is a one-to-one correspondence between $G$-supermodule structures on a finite dimensional superspace $V$ and {\it linear representations}
$G\to GL(V)$ (cf. \cite{jan, zub1}). If $\tau_V(v)=\sum v_1\otimes f_2\in V\otimes K[G]$, then $g\in G(A)$ acts on $V\otimes A$ by the even $A$-linear automorphism $\tau(g)(v\otimes 1)=\sum v_1\otimes g(f_2)$. 
In other words, $V$ is a $G$-supermodule if and only if the group functor $G$ acts on the functor $V_a$ so that for any $A\in\mathsf{SAlg}_K$ the group $G(A)$ acts on $V_a(A)=V\otimes A$ by even $A$-linear automorphisms. 

If $V$ is infinte dimensional, then a $G$-supermodule structure on $V$ is uniquely defined by a directed family of finite dimensional subrepresentations $\{\tau_i : G\to GL(V_i) | i\in I\}$, where $(I, \leq)$ is a directed set such that $V_i\subseteq V_j$ if and only if $i\leq j$, $\bigcup_{i\in I} V_i=V$ and $\tau_i |_{V_i\bigcap V_j}=
\tau_j |_{V_i\bigcap V_j}$ for all $i, j\in I$.

If $V$ is an one dimensional $G$-supermodule, then its supercomodule structure is uniquely defined by a group-like element
$f\in K[G]$ so that $\tau_V(v)=v\otimes f$. All group-like elements of $K[G]$ form a group of characters $X(G)$ of $G$.

\section{General linear supergroup}

Let $V$ be a finite dimensional superspace. The group functor $A\to \mathsf{End}_A(V\otimes A)_{0}^*$ is an algebraic supergroup. It is called a {\it general linear supergroup} and denoted by $GL(V)$. If $\dim V_0=m, \dim V_1=n$, then $GL(V)$ is also denoted by $GL(m|n)$. 

Fix a homogeneous basis of $V$, say $v_i, 1\leq i\leq m+n$, where $|v_i|=0$ provided $1\leq i\leq m$, otherwise
$|v_i|=1$.
It is easy to see that
$K[GL(m|n)]=K[c_{ij}|1\leq i, j\leq m+n]_{d}$, where $|c_{ij}|=|v_i|+|v_j|$. More precisely, the generic matrix $C=(c_{ij})_{1\leq i, j\leq m+n}$ has a block form 
$$\left(\begin{array}{cc}
C_{00} & C_{01} \\
C_{10} & C_{11}
\end{array}\right)
$$
with even $m\times m$ and $n\times n$ blocks $C_{00}$ and $C_{11}$, and odd $m\times n$ and $n\times m$ blocks
$C_{01}$ and $C_{10}$ respectively. Besides, 
$$\Delta_{GL(m|n)}(c_{ij})=\sum_{1\leq k\leq m+n} c_{ik}\otimes c_{kj}, \ \epsilon_{GL(m|n)}(c_{ij})=\delta_{ij},$$
and $d=\det(C_{00})\det(C_{11})$. The right supercomodule structure of $V$ is defined by
$$\tau_V(v_i)=\sum_{1\leq j\leq m+n} v_j\otimes c_{ji}.$$

The element $Ber(C)=\det(C_{00}-C_{01}C_{11}^{-1}C_{10})\det(C_{11})^{-1}$ is called {\it Berezinian}.
This is a group-like element of the Hopf superalgebra $K[GL(m|n)]$ (cf. \cite{ber}).

Observe that $\mathbb{Z}[GL(m|n)]=\mathbb{Z}[c_{ij}| 1\leq i, j\leq m+n]_d$ is a $\mathbb{Z}$-form of $K[GL(m|n)]$. Moreover, $\mathbb{Z}[GL(m|n)]_r$ ($\mathbb{Z}[GL(m|n)]_l$) is a
$\mathbb{Z}$-form of $K[GL(m|n)]_r$ (respectively, $K[GL(m|n)]_l$).

\section{Borel supersubgroups and root systems}

Let $G=GL(m|n)$. We fix the standard maximal torus $T$ such that $T(A)$ consists of all diagonal 
matrices from $G(A), A\in\mathsf{SAlg}_K$. Denote by $X(T)$ the group of characters of $T$.
We identify $X(T)$ with the additive group  
$\mathbb{Z}^{m+n}$. In particular, any $\lambda\in X(T)$ has a form
$\sum_{1\leq i\leq m+n}\lambda_i\epsilon_i$, where 
$$\epsilon_i(t)=t_i, t=\left(\begin{array}{cccc}
t_1 & 0 & \ldots & 0 \\
0   & t_2 & \ldots & 0 \\
\vdots & \vdots & \ldots & 0 \\
0 & 0 & \dots & t_n
\end{array}\right)\in T(A), A\in\mathsf{SAlg}_K ,$$  
and any $\lambda_i$ is an integer. 
For a character $\lambda\in X(T)$ denote $\sum_{1\leq i\leq m+n}\lambda_i$ by $|\lambda|$. 
Define a bilinear form on $X(T)\otimes_{\mathbb{Z}} \mathbb{Q}$ setting
$(\epsilon_i, \epsilon_j)=\delta_{ij}(-1)^{|v_i|}$. Let $\epsilon'_i$ denote $(-1)^{|v_i|}\epsilon_i$.
Then $(\epsilon_i, \epsilon'_j)=\delta_{ij}$.

Consider an root system $$\Phi_w=\{\epsilon_{wi}-\epsilon_{wj}| 1\leq i\neq j\leq m+n\},$$ where
$w\in S_{m+n}$. The corresponding coroots are $(\epsilon_{wi}-\epsilon_{wj})^{\vee}=\epsilon'_{wi}-\epsilon'_{wj}$.

Its positive part $$\Phi^+_w =\{\epsilon_{wi}-\epsilon_{wj}| 1\leq i < j\leq m+n\}$$
corresponds to a Borel supersubgroup $B^+_{w}$ that is the stabilizer of the full flag
$$V_1\subseteq V_2\subseteq\ldots\subseteq V_i\subseteq\ldots\subseteq V_{m+n}=V,$$
where $V_i=\sum_{1\leq s\leq i} Kv_{ws}, 1\leq i\leq m+n$. 

The opposite Borel supersubgroup $B^-_w$ corresponds to the negative
part $$\Phi_w^-=\{\epsilon_{wi}-\epsilon_{wj}| 1\leq i > j\leq m+n\}$$
of $\Phi_w$. In other words, $B^-_w$ is the stabilizer of the full flag
$$W_1\subseteq W_2\subseteq\ldots\subseteq W_i\subseteq\ldots\subseteq W_{m+n}=V,$$
where $W_i=\sum_{m+n-i+1\leq s\leq m+n} Kv_{ws}, 1\leq i\leq m+n$.

The simple roots of $\Phi^+_w$ form a subset $$\Pi_w=\{\alpha_i=\epsilon_{wi}-\epsilon_{w(i+1)}| 1\leq i < m+n\}.$$
The root system $\Phi_w$ defines a partial order $<_w$ on the weight lattice $X(T)$ by $\mu <_w\lambda$ if 
$\lambda-\mu\in\sum_{\alpha\in\Phi_w^+} \mathbb{N}_+\alpha=\sum_{\alpha\in\Pi_w} \mathbb{N}_+\alpha$.

An root $\alpha=\epsilon_i-\epsilon_j$ has a {\it parity} $p(\alpha)=|v_i|+|v_j|$. 
Denote $\{\alpha\in\Phi_w | p(\alpha)=a\}$ by $(\Phi_w)_a$, where $a=0, 1$. 

For any $\alpha\in (\Phi_w)_0$ one can define an reflection $s_{\alpha}$ such that
$s_{\alpha}(\lambda)=\lambda-(\lambda, \alpha^{\vee})\alpha$.
It is easy to see that if $\alpha=\epsilon_i-\epsilon_j$, then $s_{\alpha}=(ij)$. 
These reflections generate the {\it Weyl subgroup} $S_m\times S_n\subseteq S_{m+n}$.  

Any $w\in S_{m+n}$ can be uniquely decomposed as $w=w_0 w_1$, where $w_0\in S_m\times S_n$ and $w_1$ satisfies $w_1^{-1}1<\ldots <w_1^{-1}m, w_1^{-1}(m+1)<\ldots <w_1^{-1}(m+n)$ (cf. \cite{brunkuj}).
Then 
$$(\Phi^+_w)_0=\{\epsilon_{w_0 i}-\epsilon_{w_0 j}| 1\leq i < j\leq m, m+1\leq i< j\leq m+n\}$$ and    
$$(\Phi^-_w)_0=\{\epsilon_{w_0 i}-\epsilon_{w_0 j}| 1\leq j < i\leq m, m+1\leq j < i\leq m+n\}=
-(\Phi^+_w)_0 .$$
In other words, $(B^+_w)_{ev}$ ($(B^-_w)_{ev}$) coincides with $(B^+_{w_0})_{ev}$ ($(B^-_{w_0})_{ev}$). If $w_0=1$, then 
they are the upper triangular (respectively, lower triangular) subgroup of $G_{ev}=GL(m)\times GL(n)$.

A Borel supersubgroup $B=B^-_w$ is called {\it standard}, whenever $(\Pi_w)_0$ is a set of simple roots of
$B_{ev}$. For example, if $w=w_0 w_1$ is the above decomposition with $w_0\in S_m\times S_n$, then  $B'=B^-_{w_0}$ is standard
and $B'_{ev}=B_{ev}$.   
In general, a Borel supergroup is not necessary standard. For example, set $m=n=2$ and $w=(1342)$. Then all simple roots 
$\epsilon_3 -\epsilon_1 , \epsilon_1 -\epsilon_4 , \epsilon_4 -\epsilon_2$ are odd! 

An root $\alpha=\epsilon_i-\epsilon_j$ corresponds to the one dimensional unipotent supersubgroup 
$U_{ij}(A)=\{E+aE_{ij}| a\in A_{|\alpha|}\}, A\in\mathsf{SAlg}_K, 1\leq i\neq j\leq m+n$.
We denote $U_{ij}$ by $U_{\alpha}$ also. 

Any Borel supersubgroup $B_w^{\pm}$ is a semidirect product of the torus $T$ and its unipotent radical 
$U_w^{\pm}$. By definition, $U_w^+$ ($U_w^-$) is the largest supersubgroup of $B_w^{+}$ (respectively, of $B_w^{-}$) that acts trivially on each quotient $V_{i+1}/V_i$ (respectively, on each quotient $W_{i+1}/W_i$), $1\leq i\leq m+n-1$. 

Denote $\sum_{\alpha\in(\Phi_w^+)_0} \alpha$ by $\rho_0(w)$, and $\sum_{\alpha\in(\Phi_w^+)_1}\alpha$ by $\rho_1(w)$. 
The above remark infers that the element $\rho_0(w)$ depends of $w_0$ only and 
$$\rho_0(w)=\sum_{1\leq i\leq m}(m-2i+1)\epsilon_{w_0 i} +\sum_{m+1\leq j\leq m+n}(n-2(j-m)+1)\epsilon_{w_0 j}=w_0 \rho(id).$$
Analogously, $\rho_1(w)=w_0 \rho_1(w_1)$ but $\rho_1(w_1)$ actually depend of $w_1$. For example, 
$$\rho_1(id)=\sum_{1\leq i\leq m}n\epsilon_i -\sum_{m+1\leq j\leq m+n}m\epsilon_j.$$
But if $m=n=2$ and $w=(2 3)$, then
$$\rho_1(w)=2\epsilon_1-2\epsilon_4\neq 2(\epsilon_1+\epsilon_2-\epsilon_3-\epsilon_4).$$
Set $\rho(w)=\frac{1}{2}(\rho_0(w)-\rho_1(w))$. Define a {\it dot action} $u._w\lambda=u(\lambda+\rho(w))-\rho(w), \lambda\in X(T), u\in S_m\times S_n$.
This action depends of $w$.

\section{Induced supermodules}

If $H\leq G$ and $V$ is a left $H$-supermodule, then $ind^G_H V= V\square_{K[H]} K[G]$, where $K[G]$ is regarded as a left  $K[H]$-supercomodule via $(\pi_H\otimes id_{K[G]})\Delta_G$. Moreover, $ind^G_H V$ is a left $G$-supermodule via $(id_V \otimes\rho_G)$.  

The above definition is different from the definition given in \cite{jan}. 
More precisely, consider $K[G]$ as a left $H$-supermodule via the right comodule map $\rho_r|_{H}=(id_{K[G]}\otimes\pi_H)\Delta_G$. Then $ind^G_H V=
(V\otimes K[G])^H$, where $H$ acts diagonally on $V\otimes K[G]$. In this setting $ind^G_H V$ is a left $G$-supermodule via $(id_V\otimes\rho_l)$.  These two $G$-supermodules are naturally isomorphic via the map $id_V\otimes s_G$. 

Consider a $K$-functor $\mathfrak{Mor}(G, V_a)$ defined as 
$$\mathfrak{Mor}(G, V_a)(A)=\mathsf{Mor}(G|_{\mathsf{SAlg}_A}, V_a|_{\mathsf{SAlg}_A}), A\in\mathsf{SAlg}_K .$$
It is clear that each $\mathfrak{Mor}(G, V_a)(A)$ has natural structure of right $A$-module.
  
Observe that $G|_{\mathsf{SAlg}_A}\simeq SSp_A \ K[G]\otimes A$ and by Yoneda's lemma  
$$\mathfrak{Mor}(G, V_a)(A)\simeq V_a|_{\mathsf{SAlg}_A}(K[G]\otimes A)=V\otimes K[G]\otimes A .$$
More precisely, an element $v\otimes f\otimes a\in V\otimes K[G]\otimes A$ represents a morphism $\phi\in\mathfrak{Mor}(G, V_a)(A)$
if and only if $\phi(B)(x)=v\otimes x(f)a$ for any $x\in G(B), B\in\mathsf{SAlg}_A$.
Thus $\mathfrak{Mor}(G, V_a)$ can be identified with $(V\otimes K[G])_a$.

The supergroups $G$ and $H$ act on the functor $\mathfrak{Mor}(G, V_a)$ on the right by the rule:
$$(g\phi)(B)(x)=\phi(B)(xg), (h\phi)(B)(x)=h(\phi(B)(x)),$$ 
$$g\in G(A), x\in G(B), h\in H(A), \phi\in\mathfrak{Mor}(G, V_a)(A),$$ 
$$A\in\mathsf{SAlg}_K, B\in\mathsf{SAlg}_A.$$   
These two actions commute each other. Define a subfunctor $\mathfrak{ind}^G_H V$ of $\mathfrak{Mor}(G, V)$ such that
for any $A\in\mathsf{SAlg}_K$ an element $\phi\in\mathfrak{Mor}(G, V_a)(A)$ belongs to $\mathfrak{ind}^G_H V (A)$
if and only if 
$$\phi(B)(hg)=h(\phi(B)(g)), \forall g\in G(B), \forall h\in H(B), \forall B\in\mathsf{SAlg}_A .$$
It is clear that $\mathfrak{ind}^G_H V$ is a $G$-stable subfunctor. 
\begin{lm}\label{otherdefofind}
We have an isomorphism $(ind^G_H V)_a\simeq \mathfrak{ind}^G_H V$ that commutes wuth the action of $G$. 
\end{lm}
\begin{proof}
The above identification implies $(ind^G_H V)_a\subseteq \mathfrak{ind}^G_H V$. Considering $B=K[H]\otimes K[G]\otimes A$ and $g=1_{K[H]}\otimes id_{K[G]}\otimes 1_A, h=id_{K[H]}\otimes 1_{K[G]}\otimes 1_A$ we obtain the reverse inclusion.  
\end{proof}
Let $G$ be an affine supergroup and $H$ be a supersubgroup of $H$. Assume that there are an affine superscheme $U$ and an isomorphism of affine superschemes $\phi : G\to H\times U$ that commutes with the natural left $H$-actions on both $G$ and $H\times U$. The next lemma follows immediately by Lemma \ref{otherdefofind}. Nevertheless, we give another proof in terms of Hopf superalgebras. Denote the dual superalgebra morphism $K[H]\otimes K[U]\to K[G]$ by $\phi^*$. 
\begin{lm}\label{indisom} (see \cite{zub2}, Lemma 5.1)
For any $H$-supermodule $V$ the map $(id_V\otimes\phi^*)(\tau_V\otimes id_{K[U]})$ is a superspace isomorphism of $V\otimes K[U]$ onto $ind^G_H V$.
\end{lm}
\begin{proof}
Since $G\to H\times U$ is $H$-equivariant, it implies $$(\pi_H\otimes id_{K[G]})\Delta_G\phi^*=
(id_{K[H]}\otimes\phi^*)(\Delta_H\otimes id_{K[U]}).$$ 
Combining with (co)associativity of $\tau_V$ we see that
$V\otimes K[U]\to ind^G_H V$ is a superspace monomorphism. Conversely, if $\sum v\otimes \phi^*(h\otimes u)\in ind^G_H V$, where $v\in V, h\in K[H], u\in K[U]$, then
$$\sum v_1\otimes g_2\otimes \phi^*(h\otimes u)=\sum v\otimes h_1\otimes\phi^*(h_2\otimes u).$$
Here $\tau_V(v)=\sum v_1\otimes g_2, \Delta_H(h)=\sum h_1\otimes h_2$. Since $\phi^*$ is an isomorphism, we have
$$\sum v_1\otimes g_2\otimes h\otimes u=\sum v\otimes h_1\otimes h_2\otimes u$$
and therefore,
$$\sum v_1\otimes \phi^*(g_2\otimes (\epsilon_H(h)u)=\sum v\otimes \phi^*(h_1\otimes\epsilon_H(h_2)u)
=\sum v\otimes\phi^*(h\otimes u).
$$ 
\end{proof}
\begin{rem}\label{trivialcase}
Let $U\leq G$ and the multiplication map $m$ induces an isomorphism of affine superschemes $H\times U\to G$, then $\phi=m^{-1}$
and we identify $ind^G_H V$ with $V\otimes K[U]$ as above. Then $U$ acts on the last superspace diagonally by $\rho_r$ on $K[U]$ and trivially on $V$.
\end{rem}

\section{Quotients}

Let $G$ be an algebraic supergroup and $H\leq G$. The sheafification of the $K$-functor $A\to G(A)/H(A), A\in \mathsf{SAlg}_K,$ is called a \emph{sheaf quotient} and it is denoted by
$G / H$.  It was proved in \cite{maszub} that a quotient sheaf $G / H$ is a Noetherian superscheme and the quotient morphism $\pi : G\to X$ is affine and faithfully flat. 

Observe that $G/H$ is affine if and only if $(G/H)_{ev}=G_{ev}/H_{ev}$ is affine if and only if $G_{res}/ H_{res}$ is
(cf. \cite{maszub}, Corollary 8.15 and Proposition 9.3). In particular, $G/G_{ev}$ is always affine (see also \cite{amas}).
If $H$ is finite, then $G/H$ is also affine (combine the above criterion with \cite{jan}, I.5.5(6); see also \cite{zub3}).   

Assume that $H\unlhd G$. Then $G/ H\simeq SSp \ K[G]^H$ is an algebraic supergroup. Let $L$ be a closed subsupergroup of $G$ and $I_L$ be its
defining Hopf superideal.
A sheafification of the group subfunctor $A\to L(A)H(A)/H(A), A\in \mathsf{SAlg}_K,$ in $G/ H$, is denoted by $\pi(L)$.
It is a closed subsupergroup of $G/ H$ defined by the Hopf superideal $K[G]^H\bigcap I_L$ (cf. \cite{zub1}, Theorem 6.1). 
The closed subsupergroup $\pi^{-1}(\pi(L))$ is denoted by $LH$. As it was observed in \cite{zub1}, p.735, $LH$ is a sheafification of
the group subfunctor $A\to L(A)H(A), A\in \mathsf{SAlg}_K$. Its defining Hopf superideal coincides with $K[G](K[G]^H\bigcap I_L)$.

\section{Some standard homological properties of supergroups}

Let $G$ be an algebraic supergroup and $H$ be a subgroup of $G$. 
A $k$-th right derived functor $R^k ind^G_H ?$ is denoted by $H^k(G/H, ?)$ also.

The following theorem is a superanalog of Theorem 4.1 from \cite{cps} that is called {\it Mackey imprimitivity theorem}.       
\begin{theorem}\label{simpleimprim}
If $L\to G/H$ is an epimorphism of sheaves, then for any $H$-supermodule $V$ and any $k\geq 0$ we have an isomorphism of $L$-supermodules $$(H^k (G/H,  V)|_{L}\simeq H^k (L/{L\bigcap H}, V|_{L\bigcap H}).$$
\end{theorem}
\begin{proof}
It is clear that $\{H^k (G/H,  V)|_{L}\}_{k\geq 0}$ is a {\it $\delta$-functor} erasable by injectives (cf. \cite{lang}, chapter X, \S 7). By Theorem 7.1 and Corollary 7.2  from  \cite{lang}, chapter X, we obtain $H^k (G/H,  V)|_{L}\simeq R^k H^0 (G/H, V)|_{L}=
R^k (ind^G_H V)|_{L}$. It remains to prove that $(ind^G_H V)|_{L}\simeq ind^L_{L\bigcap H} V|_{L\bigcap H}$, or
$(\mathfrak{ind}^G_H V)|_{L}\simeq \mathfrak{ind}^L_{L\bigcap H} V|_{L\bigcap H}$. 

Consider a subfunctor $S\subseteq G$ such that $S(A)=H(A)L(A), A\in\mathsf{SAlg}_K$. Since the functor $S$ obviously commutes with direct products of superalgebras, the sheafification $\tilde{S}$ of $S$ coincides with $G$. More precisely, for any $A\in\mathsf{SAlg}_K$ and $g\in G(A)$ there is an fppf covering $B$ of $A$ such that
$G(\iota^B_A)(g)\in S(B)$, where $\iota^B_A : A\to B$ is the canonocal embedding (cf. \cite{zub1}, p.722).
Thus the restriction $\phi\mapsto \phi|_{L}$ induces an embedding $(\mathfrak{ind}^G_H V)|_{L}\to \mathfrak{ind}^L_{L\bigcap H} V|_{L\bigcap H}$. Conversely, for a given $\psi\in \mathfrak{ind}^L_{L\bigcap H} V(A)$ one can define $\phi\in\mathfrak{ind}^G_H V (A)$ such that $\phi|_L=\psi$ as follows. For any $B\in\mathsf{SAlg}_A$ and $g\in G(B)$ choose an fppf covering $C$ of $B$ such that
$G(\iota^C_B)(g)=hl, h\in H(C), l\in L(C)$. Set 
$\phi(B)(g)=V_a(\iota^C_B)^{-1}(h\psi(l))$. We leave to the reader the routine verification that the definition of $\phi(B)(g)$
does not depend on the choice of $C$ and functorial on $B$. Thus our theorem follows.
\end{proof}
Let $G$ be a normal supersubgroup of an algebraic supergroup $S$. Assume that there is a supersubgroup $L$ of $S$ such that $L\bigcap G=1$ and $S=LG$. In other words, $S$ is isomorphic to a semidirct product $G\rtimes L$. Assume also that $L$ stabilizes $H$. Then $LH\simeq H\rtimes L$. The following corollary is a (super)analog of I.4.9(1), \cite{jan}. 
\begin{cor}\label{semidirect}
For any $k\geq 0$ and any $H\rtimes L$-supermodule $V$ we have $H^k(G\rtimes L/H\rtimes L, V)|_G\simeq H^n(G/H, V|_H)$.
\end{cor}
If $char K=p>0$ and $K$ is perfect, then Theorem \ref{simpleimprim} has another interesting corollary. 
\begin{cor}\label{reductiontoeven}
For any $k\geq 0,  r\geq 1$ and any $H$-supermodule $V$ we have 
$$H^k(G/H, V)|_{G_{ev}}\simeq H^k(G_{ev}/(HG_r)_{ev}, (ind^{HG_r}_H V)|_{(HG_r)_{ev}}).$$
\end{cor}
\begin{proof}
Observe that $G/G_r\simeq SSp \ K[G]^{p^r}$ by Theorem 6.1, \cite{zub1}. Since $K[G]^{p^r}=K[G]_0^{p^r}$, $G/G_r=(G/G_r)_{ev}$
and by Proposition 9.3, \cite{maszub}, the induced morphism $G_{ev}\to G/G_r$ is an epimorphism of sheaves.
The same is true for $G_{ev}\to G/HG_r$. Since $HG_r/H\simeq G_r/H_r$ is an affine superscheme (cf. \cite{zub3}),
the functor $ind^{HG_r}_H$ is exact and the standard spectral sequence arguments infer 
$$H^k(G/H, V)\simeq H^k(G/HG_r, ind^{HG_r}_H V).$$ 
Theorem \ref{simpleimprim} concludes the proof. 
\end{proof}
The next lemma is a (super)analog of I.6.11, \cite{jan}.
\begin{lm}\label{restrictiontofactor}
Let $N\unlhd G$ and $N\leq H$. If $V$ is an $H/N$-supermodule, then 
$$H^k(G/H, V)\simeq H^k((G/N)/(H/N), V), k\geq 0 .$$
\end{lm}
\begin{proof}
The category $H/N-smod$ can be considered as a full subcategory of the category $H-smod$.
The restriction of $\{H^k(G/H, ?)\}_{k\geq 0}$ on this subcategory is a $\delta$-functor. Since $G/N$ and $H/N$ are affine, one can argue as in I.6.11, \cite{jan} to prove that $\{H^k(G/H, ?)\}_{k\geq 0}$ is erasable by injective $H/N$-supermodules.
As above, all we need is to prove that $ind^G_H V\simeq ind^{G/N}_{H/N} V$. Observe that if $W$ is an injective $N$-supermodule, then the spectral sequence in Proposition 3.1 (3), \cite{zubscal}, degenerates, that is $H^k(H/N, W^N)\simeq  
H^k(H, W)$ for any $k\geq 0$. Since $K[G]$ is an injective $N$-supermodule (cf. \cite{zub2}), we have 
$$ind^G_H V\simeq (V\otimes K[G])^H\simeq ((V\otimes K[G])^N)^{H/N}\simeq (V\otimes K[G/N])^{H/N}\simeq ind^{G/N}_{H/N} V.$$
\end{proof}
The proof of the following statement can be copied from \cite{jan}, I.4.10.
\begin{lm}\label{indandfixedpoint}
Let $H$ be a supersubgroup of an algebraic supergroup $G$. For any $H$-supermodule $V$ and any $k\geq 0$ we have a superspace isomorphism
$H^k(G/H, V)\simeq H^k(H, V\otimes K[G])$. 
\end{lm} 

\section{Representations of minimal parabolic supersubgroups}

Let $a$ be an integer. For the sake of convenience we say that $0|a$ if $a=0$ only. Respectively, $0\nmid a$ means $a\neq 0$.
From now on $char K=p$ and it is possible that $p=0$.

For any $\lambda\in X(T)$ let $K_{\lambda}^{\epsilon}$ denote the one-dimensional $B^{\pm}$-supermodule of weight $\lambda$ and of parity $\epsilon$. Sometimes we will denote a $G$-supermodules $H^k(G/B_w^{\pm}, K^{\epsilon}_{\lambda})$ by $H^k_{\pm , w}(\lambda^{\epsilon}), k\geq 0$. The simple socle of $H^0_{\pm , w}(\lambda^{\epsilon})$ is denoted by $L_{\pm, w}(\lambda^{\epsilon})$
(see \cite{brunkuj}, Lemma 4.1). Observe that $\Pi H^k_{\pm , w}(\lambda^0)=H^k_{\pm , w}(\lambda^1)$ and
$\Pi L_{\pm, w}(\lambda^0)=L_{\pm, w}(\lambda^1)$.

Let $G=GL(1|1)$. Then $\Phi^+_{id}=-\Phi^+_{(12)}=\{\alpha_1=\epsilon_1-\epsilon_2\}$ and $(\Phi_{id})_0=\emptyset$.
In other words, $G$ has only two Borel supersubgroups $B_{id}^+ =B_{(12)}^-$ and $B_{id}^- =B_{(12)}^+$.
Thus $<_1$ is opposite to $<_{(12)}$.
In what follows we omit a subindex $w\in S_2$. 

Observe that the multiplication morphism induces isomorphisms of superschemes $U_{12}\times B^- \simeq G$ and $U_{21}\times B^+ \simeq G$.
 
By Lemma 6.1, \cite{zubmar}, a $G$-supermodule
$H_-^0(\lambda^{\epsilon})$ is two-dimensional. 
More precisely, $H^0_-(\lambda^{\epsilon})_{\mu}\neq 0$ if $\mu\in\{\lambda, \lambda-\alpha\}$ only.
The one-dimensional supersubspaces $H^0_-(\lambda^{\epsilon})_{\lambda}$ and $H^0_-(\lambda^{\epsilon})_{\lambda-\alpha}$ have parities $\epsilon$ and $\epsilon+1$ respectively.  
The $G$-supermodule $H^0_-(\lambda^{\epsilon})$ is simple if and only if $p\nmid |\lambda|$, otherwise
$H^0_-(\lambda^{\epsilon})$ has a composition series
$$\begin{array}{c}
L_-((\lambda-\alpha)^{\epsilon+1}) \\
| \\
L_-(\lambda^{\epsilon})
\end{array}.
$$
Symmetrically, $H^0_+(\lambda^{\epsilon})_{\mu}\neq 0$ if $\mu\in\{\lambda, \lambda+\alpha\}$ only.
The one-dimensional supersubspaces $H^0_+(\lambda^{\epsilon})_{\lambda}$ and $H^0_+(\lambda^{\epsilon})_{\lambda+\alpha}$ have parities $\epsilon$ and $\epsilon+1$ respectively.  
The $G$-supermodule $H^0_+(\lambda^{\epsilon})$ is simple if and only if $p\nmid |\lambda|$. If it is the case, then
$H^0_+(\lambda^{\epsilon})\simeq H^0_-((\lambda+\alpha)^{\epsilon+1})$. Otherwise
$H^0_+(\lambda^{\epsilon})$ has a composition series
$$\begin{array}{c}
L_+((\lambda+\alpha)^{\epsilon+1})=L_-((\lambda+\alpha)^{\epsilon+1}) \\
| \\
L_+(\lambda^{\epsilon})=L_-(\lambda^{\epsilon})
\end{array}.
$$
Now, let $G=GL(m|n)$. For any subset $S\subseteq\Pi_w$ 
one can define a parabolic supersubgroup $P_w(S)$. More precisely, if $S=\{\alpha_{i_1},\ldots , \alpha_{i_r}|
1\leq i_1 <\ldots < i_r <m+n\}$, then $P_w(S)$ is equal to the stabilizer of the flag
$$W_1\subseteq\ldots\subseteq W_{m+n-i_r-1}
\subseteq W_{m+n-i_r+1}\subseteq\ldots$$
$$\subseteq W_{m+n-i_1 -1}\subseteq W_{m+n-i_1+1}\subseteq\ldots\subseteq W_{m+n}.$$
For example, if $S=\{\alpha_i\}$, then $P_w(\alpha_i)$ coincides with the stabalizer of the flag
$$W_1\subseteq\ldots\subseteq W_{m+n-i-1}\subseteq W_{m+n-i+1}\subseteq\ldots\subseteq W_{m+n}.$$ 

Denote a supersubspace $Kv_{wi}+Kv_{w(i+1)}$ by $S_i$. Define a supersubgroup $H_{i,w}$ of $G$ such that  
for all $A\in\mathsf{SAlg}_K$ :
$$H_{i, w}(A)=\{g\in G(A) | g(S_i\otimes 1)\subseteq S_i\otimes A, g(K v_j\otimes 1)\subseteq K v_j\otimes A, j\neq wi, w(i+1)\}.$$
It is clear that $H_{i, w}\simeq GL(2)\times T'$ whenever $\alpha_i$ is even, otherwise $H_{i, w}\simeq GL(1|1)\times T'$, where
$T'(A)=\{t\in T(A)| t|_{S_i\otimes 1}=id_{S_i\otimes 1}\}$.
  
Let $UP_w(\alpha_i)$ be a largest supersubgroup of $U_w^-$ whose elements act trivially on $W_{m+n-i+1}/W_{m+n-i-1}$.
The following lemma is obvious.
\begin{lm}\label{Levidecomp}
We have $P_w(\alpha_i)=UP_w(\alpha_i)\rtimes H_{i, w}$.  
\end{lm} 
\begin{rem}\label{noticeaboutunip}
One can prove that $UP_w(\alpha_i)$ is the unipotent radical of $P_w(\alpha_i)$, that is the largest 
connected, normal and unipotent supersubgroup of $P_w(\alpha_i)$ (see \cite{amas, zub3}). 
\end{rem}
We say that two Borel supergroups $B^-_w$ and $B^-_{w'}$ are {\it adjuacent} via $\alpha_i\in\Pi_w$ if  $\Phi^+_{w'}=\Phi_w^+\setminus\{\alpha_i\}\bigcup\{-\alpha_i\}$. If $\alpha_i$ is odd, then we say that
$B^-_w$ and $B^-_{w'}$ are {\it odd adjuacent}, otherwise they are {\it even adjuacent}.
It is clear that $UP_w(\alpha_i)\leq B^-_w, B^-_{w'}\leq P_w(\alpha)$. Moreover, as in Lemma \ref{Levidecomp} we have
$$B^-_w=UP_w(\alpha_i)\rtimes (B^-\times T') , \ B^-_{w'}=UP_w(\alpha_i)\rtimes (B^+\times T').
$$
Here $B^-$ and $B^+$ are the corresponding Borel subgroups of $GL(1|1)$ or of $GL(2)$, with respect to the parity of $\alpha_i$.

In what follows we omit the subindexes $w, w'$ and $i$.
For example, $B_w^-$ and $B_{w'}^-$ are denoted just by $B$ and $B'$ correspondingly, $P_w(\alpha_i)$ by $P(\alpha)$ etc. 
\begin{pr}\label{indover} If $\alpha\in \Pi_1$, then :  
\begin{enumerate} 
\item For any $B$-supermodule $M$ we have $H^k (P(\alpha)/B, M)=0$, provided $k > 0$.
\item For all $k\geq 0$ we have $H^k (G/B, M)\simeq H^k (G/P(\alpha), \ ind^{P(\alpha)}_B M)$.
\end{enumerate}
\end{pr}
\begin{proof}
Observe that the multiplication map $U_{\alpha}\times B\to P(\alpha)$ induces an isomorphism of superschemes. Thus
$P(\alpha)/B$ is an affine superscheme and the functor  $ind^{P(\alpha)}_B$ is exact (cf. \cite{zub1, zub2}). In particular, $H^k(P(\alpha)/B, M)=0$ for all $k>0$. The second statement follows by the standard spectral sequence arguments.
\end{proof}
\begin{rem}\label{remaboutadj}
All statements of the above proposition hold for $B'$ also.
\end{rem}

Denote the socle of $ind^{P(\alpha)}_B K_{\lambda}^{\epsilon}$ by $L_P(\lambda^{\epsilon})$.
\begin{pr}\label{reprofaddroot} Let $\alpha\in\Pi_1$. 
\begin{enumerate}
\item $UP(\alpha)$ acts trivially on $ind^{P(\alpha)}_B K_{\lambda}^{\epsilon}$.
\item If $p\nmid (\lambda, \alpha)$, then $L_P(\lambda^{\epsilon})=ind^{P(\alpha)}_B K_{\lambda}^{\epsilon}\simeq ind^{P(\alpha)}_{B'} K_{\lambda+\alpha}^{\epsilon+1}$ is a simple $P(\alpha)$-supermodule of highest weight $\lambda$ (with respect to $<_w$).
\item If $p |(\lambda, \alpha)$, then $ind^{P(\alpha)}_B K_{\lambda}^{\epsilon}$ has a composition series
$$\begin{array}{c}
L_P((\lambda-\alpha)^{\epsilon+1}) \\
| \\
L_P(\lambda^{\epsilon})
\end{array}.
$$
\item  If $p |(\lambda, \alpha)$, then $ind^{P(\alpha)}_{B'} K_{\lambda}^{\epsilon}$ has a composition series
$$\begin{array}{c}
L_P((\lambda+\alpha)^{\epsilon+1}) \\
| \\
L_P(\lambda^{\epsilon})
\end{array}.
$$
\end{enumerate}
\end{pr}
\begin{proof}
The first statement follows by Lemma \ref{restrictiontofactor}. Moreover, 
we have an isomorphism $ind^{P(\alpha)}_B K_{\lambda}^{\epsilon}\simeq
ind^{GL(1|1)\times T'}_{B^-\times T'} K_{\lambda}^{\epsilon}.$ The last $GL(1|1)$-supermodule is isomorphic to
$H^0_-(\lambda^{\epsilon})$ by Corollary \ref{semidirect}. On the other hand, Lemma \ref{indisom} infers $ind^{P(\alpha)}_B K_{\lambda}^{\epsilon}\simeq K^{\epsilon}_{\lambda}\otimes K[U_{\alpha}]$. Using Lemma \ref{otherdefofind} one sees that
$T$ acts on $K[U_{\alpha}]$ by $\nu_l$. In particular, $(ind^{P(\alpha)}_B K_{\lambda}^{\epsilon})_{\mu}\neq 0$
if and only if $\mu\in\{\lambda, \lambda-\alpha\}$. Besides, the corresponding weight supersubspaces are one-dimensional and have parities $\epsilon$ and $\epsilon+1$ respectively. Thus follows the second and third statements. The case of $B'$ is symmetric. 
\end{proof}

\section{The Borel-Bott-Weil theorem}
We still suppose that $B=B^-_w$ for a fixed element $w\in S_{m+n}$. 
\begin{pr}\label{adjuacentfortypical}
Assume that $B$ and $B'$ are adjacent via $\alpha\in \Pi_1$ and $p\nmid (\lambda, \alpha)$. Then 
$H^k (G/B, K^{\epsilon}_{\lambda})\simeq  H^k (G/B' , K^{\epsilon+1}_{\lambda-\alpha})$ for all $k\geq 0$. 
\end{pr}
\begin{proof}
Combine Proposition \ref{indover} (2) with Proposition \ref{reprofaddroot} (2).
\end{proof}
\begin{pr}\label{evenroot}
Let $\alpha\in \Pi_0 , \lambda\in X(T)$. The following statements hold for any $k\geq 0$. 
\begin{enumerate}
\item The unipotent radical $UP(\alpha)$ acts trivially on $H^k (P(\alpha)/B,  K_{\lambda})$.
\item If $(\lambda, \alpha^{\vee})=-1$, then $H^k (G/B,  K_{\lambda}^{\epsilon}) =0$.
\item If $(\lambda, \alpha^{\vee})\geq 0$, then $H^k (G/B, K_{\lambda}^{\epsilon})\simeq H^k (G/P(\alpha) , ind^{P(\alpha)}_B K_{\lambda}^{\epsilon})$.
\item If $(\lambda, \alpha^{\vee})\leq -2$, then $H^k (G/B , K_{\lambda}^{\epsilon})\simeq H^{k-1} (G/P(\alpha) , H^1 (P(\alpha)/B , K_{\lambda}^{\epsilon}))$.
\item Suppose that $(\lambda, \alpha^{\vee})\geq 0$. If $char K=0$ or $char K=p > 0$ and $(\lambda, \alpha^{\vee})=sp^m -1, s, m\in\mathbb{N}, 0<s < p$, then $H^k (G/B,  K_{\lambda}^{\epsilon})\simeq H^{k+1} (G/B, K_{s_{\alpha}.\lambda}^{\epsilon})$ 
\end{enumerate}
\end{pr}
\begin{proof}
Using Lemma \ref{restrictiontofactor}, Corollary \ref{semidirect}, Lemma \ref{indisom} and Lemma \ref{otherdefofind} 
one can repeat the proofs of Propositions II.5.2, II.5.3 and II.5.4 from \cite{jan}, per verbatim. 
For example, let us comment the last statement. By Proposition 1.28, \cite{chengwang}, $(\rho, \beta)=\frac{1}{2}(\beta, \beta)$ for any simple root $\beta\in\Phi^+$. Thus 
$s_{\alpha}.\lambda=s_{\alpha}(\lambda+\rho)-\rho=\lambda-((\lambda, \alpha^{\vee})+1)\alpha=s_{\alpha}\lambda-\alpha$ (cf. \cite{jan}, II.5.1(2)).  
In particular, $(s_{\alpha}.\lambda, \alpha^{\vee})=-(\lambda, \alpha^{\vee})-2\leq -2$. By the fourth statement 
$$H^{k+1} (G/B, K_{s_{\alpha}.\lambda}^{\epsilon})\simeq H^k (G/P(\alpha), H^1 (P(\alpha)/B, K_{s_{\alpha}.\lambda}^{\epsilon})).$$
Combining the first statment with the arguments from Proposition II.5.2 and Corollary II.5.3, \cite{jan}, one sees that 
$ind^{P(\alpha)}_B K^{\epsilon}_{\lambda}\simeq H^1 (P(\alpha)/B, K^{\epsilon}_{s_{\alpha}.\lambda})$. The third statement concludes the proof.
\end{proof}
Remind that a weight $\lambda\in X(T)$ is called {\it typical} if for any odd isotropic root $\alpha$ we have
$p\nmid (\lambda+\rho, \alpha)$.  
Otherwise, $\lambda$ is called {\it atypical} (cf. \cite{chengwang}, p.62). In our case, when $G=GL(m|n)$, all roots from $\Phi^+_1$ are isotropic. Thus
$\lambda$ is typical if and only if $p\nmid (\lambda+\rho, \alpha)$ 
for any $\alpha\in \Phi^+_1$.

Let $B'=B^-_{w'}$ be a Borel supersubgroup that is adjuacent with $B$. Recall that $\rho(w')$ is denoted by $\rho'$ and
$u._{w'}\lambda$ by $u.'\lambda$.  
\begin{pr}\label{passageviaodd}
If $u.\lambda$ is typical and $B$ is odd adjuacent with $B'$ via $\alpha\in\Pi_1$, then $H^k (G/B, K^{\epsilon}_{u.\lambda})\simeq H^k (G/B' , K^{\epsilon+1}_{u.'(\lambda-\alpha)})$ for any $k\geq 0$. Besides, $u.'(\lambda-\alpha)$ is typical (with respect to $\Phi'$).     
\end{pr}
\begin{proof}
Observe that if $\mu$ is typical with respect to $\Phi$, then $\mu-\alpha$ is typical with respect to $\Phi'$. In fact, $\rho'=\rho+\alpha$ (cf. \cite{chengwang}, Proposition 1.28) and $\mu-\alpha+\rho'=\mu+\rho$. 
Since $u.'(\lambda-\alpha)=u.\lambda-\alpha$, the second statement follows. It remains to refer to Proposition \ref{adjuacentfortypical}. 
\end{proof}
For any two Borel supersubgroups $B$ and $B'$ whose even parts are the same, there is a sequence of Borel supersubgroups
$B=B^{(1)}, B^{(2)}, \ldots, B^{(t)}=B'$ such that $B^{(k)}, B^{(k+1)}$ are odd adjuacent, $1\leq k\leq t-1$.   
We say that $B'$ is a $t$ {\it steps neighbor} of $B$.
\begin{cor}\label{adjuacentviamany}
If $B'$ is a $t$ steps neighbor of $B$ and $u.\lambda$ is typical, then $$H^k (G/B, K^{\epsilon}_{u.\lambda})\simeq H^k (G/B' , K^{\epsilon+t-1}_{u.'(\lambda-(\rho'-\rho))})$$ 
for any $k\geq 0$.
\end{cor}
Define a subset $X(T)^+_w=\{\lambda\in X(T) | (\lambda, \beta^{\vee})\geq 0, \ \forall \beta\in (\Phi_w^+)_0\}$.
The elements of $X(T)_w^+$ are called {\it dominant weights} with respect to a given Borel supersubgroup $B=B^-_w$ (or, with respect to a given root system $\Phi_w$). 

It is clear that $X(T)_w^+ =X(T)^+_{w_0}$. In other words, if $B$ and $B'$ have the same even parts, then their sets of dominant weights are the same too. Following our conventions we further omit the subindex $w$.

If $char K=0$, then $C_{\mathbb{Z}}$ ($\overline{C}_{\mathbb{Z}}$) consists of all $\lambda\in X(T)$ such that $0< (\lambda+\rho, \alpha^{\vee})$ for all $\alpha\in \Phi^+_0$ (respectively, $0\leq (\lambda+\rho, \alpha^{\vee})$ for all $\alpha\in \Phi^+_0$). If $char K=p>0$, then $C_{\mathbb{Z}}$ ($\overline{C}_{\mathbb{Z}}$) consists of all $\lambda\in X(T)$ such that $0 < (\lambda+\rho, \alpha^{\vee})\leq p$ for all $\alpha\in \Phi^+_0$ (respectively, $0\leq (\lambda+\rho, \alpha^{\vee})\leq p$ for all $\alpha\in \Phi^+_0$) (cf. \cite{jan}, II.5.5).   
\begin{rem}\label{asinpenkov'sarticle}
In \cite{pen} the sets $C_{\mathbb{Z}}$ and $\overline{C}_{\mathbb{Z}}$ are denoted by $C^{+}$ and $\overline{C}^{+}$ correspondingly.
\end{rem}
Let $u\in S_m\times S_n$. The length of $u$ is denoted by
by $l(u)$. Untill Theorem \ref{BorelBottWeylpPenkov} we suppose that $B$ is standard.
\begin{lm}\label{partialcase}
Assume that $\lambda\in\overline{C}_{\mathbb{Z}}\bigcap X(T)^+$ and $k>l(u)$. Then 
$$H^k (G/B, K^{\epsilon}_{u.\lambda})\simeq H^{k-l(u)} (G/B,  K^{\epsilon}_{\lambda}).$$
\end{lm}
\begin{proof}
If $l(u)>0$, then there is $\alpha\in \Pi_0$ such that $l(s_{\alpha}u)=l(u)-1$. Arguing as in \cite{jan}, Corollary II.5.5, we obtain 
$$(s_{\alpha}u.\lambda, \alpha^{\vee})=(\lambda+\rho, \beta^{\vee})-1, \beta=-u^{-1}(\alpha)\in\Phi^+_0.$$ 
Since $\lambda\in \overline{C}_{\mathbb{Z}}\bigcap X(T)^+$, we have $1\leq (\lambda+\rho, \beta^{\vee})\leq p$ and therefore, 
$0\leq (s_{\alpha}u.\lambda, \alpha^{\vee})\leq p-1$. 
The fifth statement of Proposition \ref{evenroot} infers
$$H^{k-1} (G/B, K^{\epsilon}_{s_{\alpha}u.\lambda})\simeq H^k (G/B, K^{\epsilon}_{u.\lambda}).$$ 
Lemma follows by induction on $l(u)$. 
\end{proof}
\begin{pr}\label{partialcasecor}
If $\lambda\in\overline{C}_{\mathbb{Z}}\bigcap X(T)^+$, then $H^k (G/B, K^{\epsilon}_{\lambda})=0$ for any
$k\geq 1$.
\end{pr}
\begin{proof}
Let $u_0$ be a longest element of $S_m\times S_n$ (with respect to $\Phi^+_0$). Notice that $l(u_0)=
\dim G_{ev}/B_{ev}=\frac{m(m-1)}{2}+\frac{n(n-1)}{2}$.
Combining Grothendieck's vanishing theorem from \cite{br}, p.3, with Proposition 9.3 from \cite{maszub}, we see that
$H^{k+l(u_0)} (G/B, K^{\epsilon}_{u_0 .\lambda})=0$. It remains to refer to Lemma \ref{partialcase}.
\end{proof}
\begin{theorem}\label{BorelBottWeyl}
The following statements hold for any $k\geq 0, u\in S_m\times S_n$. 
\begin{enumerate}
\item If $\lambda\in\overline{C}_{\mathbb{Z}}\setminus X(T)^+$, then $H^k (G/B, K^{\epsilon}_{u.\lambda})=0$. 
\item If $\lambda\in\overline{C}_{\mathbb{Z}}\bigcap X(T)^+$, then 
$$H^k (G/B, K^{\epsilon}_{u.\lambda})\simeq\{\begin{array}{c} 
0, k\neq l(u), \\
H^0(G/B, K^{\epsilon}_{\lambda}), k=l(u).
\end{array}$$
\end{enumerate}
\end{theorem}
\begin{proof}   
Let $u=1$. If $\lambda\in\overline{C}_{\mathbb{Z}}\setminus X(T)^+$, then one can argue as in Corollary II.5.5, \cite{jan}. The second statement is proved in Proposition \ref{partialcasecor}.

Assume that $l(u)>0$. Again, there is $\alpha\in\Pi_0$ such that $l(s_{\alpha}u)=l(u)-1$. As in Lemma \ref{partialcase}
we have either $0\leq (s_{\alpha}u.\lambda, \alpha^{\vee})\leq p-1$ and then $$H^k (G/B, K^{\epsilon}_{u.\lambda})\simeq 
H^{k-1} (G/B, K^{\epsilon}_{s_{\alpha}u.\lambda})$$ or $(s_{\alpha}u.\lambda, \alpha^{\vee})=-1=(u.\lambda, \alpha^{\vee})$ and then
$$H^k (G/B, K^{\epsilon}_{u.\lambda})=
H^{k-1} (G/B, K^{\epsilon}_{s_{\alpha}u.\lambda})=0$$ by the second statement of Proposition \ref{evenroot}.
Our theorem follows by induction on $l(u)$.  
\end{proof}
Combining all of the above, we obtain the following characteristic free variant of Penkov's theorem (compare with
\cite{pen}, Theorem 1).
\begin{theorem}\label{BorelBottWeylpPenkov}
Let $B$ be a (not necessary standard) Borel supersubgroup of $G$ and $B'$ be a standard $t$ steps neighbor of $B$. If $u.\lambda$ is typical (with respect to $B$), then :
\begin{enumerate}
\item $H^k (G/B, K^{\epsilon}_{u.\lambda})=0$, whenever $\lambda\in\overline{C}_{\mathbb{Z}}\setminus
C_{\mathbb{Z}}$.
\item $H^k (G/B, K^{\epsilon}_{u.\lambda})=0$, whenever $\lambda\in C_{\mathbb{Z}}$ but $k\neq l(u)$.
\item $H^k (G/B, K^{\epsilon}_{u.\lambda})\simeq ind^G_{B'} K^{\epsilon+t-1}_{\lambda-(\rho'-\rho)}$, whenever
$\lambda\in C_{\mathbb{Z}}$, $k =l(u)$.
\end{enumerate}   
\end{theorem}
\begin{proof}
The conditions $\lambda\in\overline{C}_{\mathbb{Z}}\setminus C_{\mathbb{Z}}$ and 
$\lambda\in C_{\mathbb{Z}}$ are invariant with respect to a passage from $B$ to $B''$ and from $\lambda$ to $\lambda-\alpha$ respectively, where $B''$ is adjuacent to $B$ via $\alpha\in\Pi_1$. For example, the first condition is equivalent to the conditions :
\begin{enumerate} 
\item $0\leq (\lambda+\rho, \beta^{\vee})\leq a$ for any $\beta\in\Phi_0^+$;
\item there is $\gamma\in\Phi_0^+$ such that $(\lambda+\rho, \gamma^{\vee})=0$.
\end{enumerate}
Here $a=p$ whenever $char K=p>0$, otherwise $a=+\infty$. It remains to notice that $\lambda-\alpha+\rho''=\lambda+\rho$ and $\Phi^+_0=(\Phi'')_0^+$. Therefore, one can assume that $B$ is a standard Borel supersubgroup. Observe that
if $B$ is standard, then $\overline{C}_{\mathbb{Z}}\setminus C_{\mathbb{Z}}=\overline{C}_{\mathbb{Z}}\setminus X(T)^{+}$ 
and $C_{\mathbb{Z}}=\overline{C}_{\mathbb{Z}}\bigcap X(T)^{+}$. Theorem \ref{BorelBottWeyl} concludes the proof.
\end{proof}
\begin{pr}\label{simpleness}
If $\lambda\in C_{\mathbb{Z}}$ and $\lambda$ is typical, then $H^0(G/B, K^{\epsilon}_{\lambda})$ is simple.
\end{pr}
\begin{proof}
Arguing as in the above theorem and using Corollary \ref{adjuacentviamany} one can suppose that
$B$ is standard and even more, $B=B^-_{id}$. It remains to combine Theorem 1 from \cite{marko} with Corollary II.5.6 from \cite{jan}.
\end{proof}

\section{Euler characteristics}

For any $G$-supermodule $V$ let $[V]$ denote its isomorphism class in the Grothendieck ring $\mathcal{K}(G)$ of the abelian category $G-smod$.
If $H\leq G$ and $V$ is an $H$-supermodule, then one can define the {\it Euler characteristic}
$\chi(H, V)=\sum_{k\geq 0}(-1)^{k}[H^k(G/H, V)]$.
In the partial case $H=B, V=K^{\epsilon}_{\lambda}$, denote $\chi(B, K^{\epsilon}_{\lambda})$ by $\chi(B, \lambda^{\epsilon})$. 
\begin{pr}\label{adjuacentforatypical}
If $B$ and $B'$ are adjacent via $\alpha\in \Pi_1$, then
$\chi(B, \lambda^{\epsilon})=\chi(B', (\lambda-\alpha)^{\epsilon+1})$.
\end{pr}
\begin{proof}
The case $p |(\lambda, \alpha)$ should be considered only. Combining Proposition \ref{indover} (2) with Proposition \ref{reprofaddroot} (2, 3) we obtain
$$\chi(B, \lambda^{\epsilon})=\chi(G/P(\alpha), L_P((\lambda-\alpha)^{\epsilon+1}))
+\chi(G/P(\alpha), L_P(\lambda^{\epsilon}))=\chi(B', (\lambda-\alpha)^{\epsilon+1}).$$
\end{proof}
As in \cite{grusserg} we introduce a parity function $p : X(T)\to \mathbb{Z}_2$ such that $p(\lambda+\alpha)=p(\lambda)+
p(\alpha), \lambda\in X(T), \alpha\in\Phi$. This definition depends of $\Phi=\Phi_w$ and it is not unique in general (even if 
$w$ is fixed). Let $\mathcal{F}$ denotes the full
subcategory of $G-smod$ consisting of finite dimensional supermodules such that the parity of any weight space coincides
with the parity of the corresponding weight. It is not hard to see that $G-smod = \mathcal{F}\oplus\Pi\mathcal{F}$.

Let $\mathcal{R} = \mathbb{Z}[e^{\lambda}]$ for all $\lambda\in X(T)$. For each $M\in \mathcal{F}$ one can define
the character
$$ch(M) =\sum_{\mu\in X(T)}\dim M_{\mu}e^{\mu}\in\mathcal{R}.
$$
Denote by $\mathcal{K}(\mathcal{F})$ the Grothendieck ring of the category $\mathcal{F}$. 
Since the map $ch$ is additive on the short exact sequences and multiplicative on the tensor products, $ch : \mathcal{K}(\mathcal{F})\to\mathcal{R}$ is a homomorphism of rings. Besides, 
the elements $ch(L(\lambda))$ are linearly independent. Our convention about parity implies that for any
two $G$-supermodules $M$ and $N$ in $\mathcal{F}$, $ch(M) = ch(N)$ if and only if $[M] = [N]$ in $\mathcal{K}(\mathcal{F})$. Hence $ch$ is injective. 
\begin{pr}\label{characterformula}
For any $\lambda\in X(T)$ we have the following character formula :
$$ch(\chi(B, \lambda^{\epsilon}))=\frac{\prod_{\alpha\in\Phi^+_1}(1+e^{-\alpha})}{\prod_{\alpha\in\Phi^+_0}(1-e^{-\alpha})}
\sum_{u\in W}(-1)^ue^{u.\lambda}$$
\end{pr}
\begin{proof}
Observe that the above formula is invariant with respect to any passage from $B$ to an odd adjuacent $B'$.
In particular, Proposition \ref{adjuacentforatypical} reduces the general case to the case when $B$ is standard, say $B=B^-_{w_0}$. Then $B\leq P=P(\Pi_0)=Stab_G(\sum_{m+1\leq i\leq m+n}Kv_i)$. 
Denote by $U$ the unipotent supersubgroup that consists of all matrices of the form
$$\left(\begin{array}{cc}
E_m & * \\
0 & E_n
\end{array}\right). 
$$
Then the multiplication map $U\times P\to G$ induces an isomorphism of superschemes that infers 
$H^k(G/B, V)\simeq ind^G_P \ H^k(G/P, V)$ for any $B$-supermodule $V$ (cf. \cite{zub2}).
By Lemma \ref{indisom} $ind^G_P \ H^k(P/B, V)\simeq H^k(P/B, V)\otimes K[U]$ and the maximal torus $T$ acts on $U$ by
conjugations so that $ch(K[U])=\prod_{\alpha\in\Phi^+_1}(1+e^{-\alpha})$. 
Since $P=G_{ev}B$, by Theorem \ref{simpleimprim} we have an isomorphism of $G_{ev}$-modules $H^k(P/B, V)|_{G_{ev}}\simeq
H^k(G_{ev}/B_{ev}, V|_{B_{ev}})$. By Proposition II.5.10, \cite{jan}, it remains to verify that 
$$\prod_{\alpha\in\Phi_0^+}(e^{\alpha/2}-e^{-\alpha/2})=\sum_{u\in W}(-1)^{u}e^{u\rho_0}.
$$
We leave this elementary exercise for the reader.
\end{proof}
The following proposition is a partial generalization of Kempf's vanishing theorem (cf. \cite{jan}). 
\begin{pr}\label{Kempf}
Let $\lambda$ be a typical weight. If $(\lambda+\rho, \beta^{\vee})\geq 1$ for any $\beta\in\Phi^+_0$, then $H^k(G/B, K^{\epsilon}_{\lambda})=0$ for any $k\geq 1$.
\end{pr}
\begin{proof}
Arguing as in Theorem \ref{BorelBottWeylpPenkov} one can assume that $B$ is standard and even more, up to some inner automorphism of $G$, one can assume that $B=B^-_{id}$. It remains to refer to Theorem 5.1 from \cite{zub2}. 
\end{proof}
\begin{cor}\label{characterofinduced}
If $\lambda$ satisfies the conditions of the above proposition, then
$$ch(H^0(G/B, K^{\epsilon}_{\lambda}))=\frac{\prod_{\alpha\in\Phi^+_1}(1+e^{-\alpha})}{\prod_{\alpha\in\Phi^+_0}(1-e^{-\alpha})}
\sum_{u\in W}(-1)^ue^{u.\lambda}$$
\end{cor}
\begin{lm}\label{characters}
The following alternatives hold :
\begin{enumerate}
\item If $char K=0$, then $X(G)$ is generated by $Ber(C)$. 
\item If $char K=p>0$, then $X(G)$ is generated by $Ber(C)$ and $\det(C_{11})^p$.
\end{enumerate}
\end{lm}
\begin{proof}
Consider an $f\in X(G)$. It is clear that $f$ is an invariant with respect to the adjoint (co)action $\nu_l$.
Let $V$ be an one dimensional (simple) $G$-supermodule that corresponds to $f$. We have $ch(V)=e^{\lambda}$, where
$\lambda=f|_T\in X(T)$. Moreover, $p|(\lambda_i+\lambda_j)$ for all pairs $i, j$ such that $1\leq i\leq m<j\leq m+n$
(cf. \cite{grmarzub, ktr, stem}).
In other words, $\lambda=(\underbrace{a,\ldots, a}_m , \underbrace{-a+pt,\ldots , -a+pt}_n)$. It remains to observe that
the group-like element $g=Ber(C)^a\det(C_{11})^{pt}$ defines a simple $G$-supermodule of the same highest weight.
Thus $f=g$.  
\end{proof}

\section{Kempf's vanishing theorem for arbitrary weights}

\begin{lm}\label{distmap}
For any $r\geq 1,$ $m : B_r\times U^+_r\to G_r$ is an isomorphism of superschemes. 
\end{lm} 
\begin{proof}
Since $\ker\epsilon_{G_r}$ and $K[B_r]\otimes \ker\epsilon_{U^+_r}+\ker\epsilon_{B_r}\otimes K[U^+_r]$ are nilpotent,
the morphism of superspaces $dm : \mathsf{Dist}(B_r)\otimes \mathsf{Dist}(U^+_r)\to 
\mathsf{Dist}(G_r)$, induced by the multiplication of $\mathsf{Dist}(G_r)$, can be identified with 
$K[B_r\times U^+_r]^*\to K[G_r]^*$. It remains to prove that $dm$ is an isomorphism.

Arguing as in \cite{shuw}, we see that $\mathsf{Dist}(G_r), \mathsf{Dist}(B_r)$ and $\mathsf{Dist}(U^+_r)$ have the basises 
$$\prod_{\epsilon_k-\epsilon_s\in\Phi_0^-} E_{ks}^{(m_{ks})} \prod_{\epsilon_k-\epsilon_s\in\Phi_1^-} E_{ks}^{\eta_{ks}}\prod_{1\leq i\leq m+n}\left(\begin{array}{c}
E_{ii} \\
t_i
\end{array}\right)\prod_{\epsilon_k-\epsilon_s\in\Phi_0^+} E_{ks}^{(n_{ks})} \prod_{\epsilon_k-\epsilon_s\in\Phi_1^+} E_{ks}^{\theta_{ks}},
$$
$$\prod_{\epsilon_k-\epsilon_s\in\Phi_0^-} E_{ks}^{(m_{ks})} \prod_{\epsilon_k-\epsilon_s\in\Phi_1^-} E_{ks}^{\eta_{ks}}\prod_{1\leq i\leq m+n}\left(\begin{array}{c}
E_{ii} \\
t_i
\end{array}\right)
$$
and
$$\prod_{\epsilon_k-\epsilon_s\in\Phi_0^+} E_{ks}^{(n_{ks})} \prod_{\epsilon_k-\epsilon_s\in\Phi_1^+} E_{ks}^{\theta_{ks}},
$$
respectively, 
where $0\leq t_i, 0\leq n_{ks}, m_{ks} < p^r, \eta_{ks}, \theta_{ks}\in \{0, 1\}$. Thus our statement obviously follows.  
\end{proof}
\begin{lm}\label{moredistmap}
For any $r\geq 1,$ $m : B\times U^+_r\to BG_r$ is an isomorphism of superschemes.
\end{lm}
\begin{proof}
Since $U^+_r\bigcap B=1$, one can identify $B\times U^+_r$ with a subfunctor of a $K$-sheaf $BG_r$. By Lemma \ref{distmap} 
$BG_r$ is a sheafification of $B\times U^+_r$. Since $B\times U^+_r$ is an affine superscheme, its sheafification coincides with itself.
(cf. \cite{zub1}, p.721, or \cite{jan}, p.68).
\end{proof}
Lemma \ref{moredistmap} also infers that $(BG_r)_{ev}=B_{ev} G_{ev, r}\simeq B_{ev}\times U^+_{ev, r}$. Moreover, by
Corollary \ref{reductiontoeven} we have an isomorphism 
$$H^k(G/B, K^{\epsilon}_{\lambda})|_{G_{ev}}\simeq H^k(G_{ev}/B_{ev}G_{ev, r}, (ind^{BG_r}_B K^{\epsilon}_{\lambda})|_{B_{ev}G_{ev, r}}).$$ Combining Lemma \ref{indisom} with Lemma \ref{moredistmap} we see that
$ind^{BG_r}_B K^{\epsilon}_{\lambda}\simeq K^{\epsilon}_{\lambda}\otimes K[U^+_r]$.
\begin{pr}\label{filtration}
The $B_{ev}G_{ev, r}$-module $ind^{BG_r}_B K^{\epsilon}_{\lambda}$ has a filtration with quotients
that are isomorphic to $ind^{B_{ev}G_{ev, r}}_{B_{ev}} K_{\lambda-\pi}$, where $\pi$ runs over sums of roots $\alpha\in\Phi_1^+$ without repetitions.
\end{pr}
\begin{proof}
The affine superscheme $BG_r/B\simeq SSp \ K[U^+_r]$ is obviously decomposable. Since the torus $T$ acts on $U^+_r$ by
conjugations, our statement immediately follows by Thorem 2.7, \cite{br}. 
\end{proof}
\begin{lm}\label{complex}
Let $C^{\bullet}$ be a complex consisting of free $\mathbb{Z}$-modules. If for some $k\geq 0$ there is a field $K$ of zero characteristic such that $H^k(C^{\bullet}\otimes_{\mathbb{Z}} K)\neq 0$, then
$H^k(C^{\bullet}\otimes_{\mathbb{Z}} F)\neq 0$ for any field $F$. 
\end{lm}
\begin{proof}
Consider the fragment $C_{k-1}\stackrel{d_{k-1}}{\to} C_k\stackrel{d_k}{\to} C_{k+1}$. Since a subgroup of free abelian group is free, for any $t$ we have $C_t=\ker d_t\oplus V_t$ and $d_t$ induces an isomorphism $V_t\to  \mathsf{Im} \ d_t$. Thus $C_t\otimes_{\mathbb{Z}} F=\ker d_t\otimes_{\mathbb{Z}} F\oplus V_t\otimes_{\mathbb{Z}} F$ for any field $F$.

Notice that $K$ is a flat $\mathbb{Z}$-modue. Thus
$$H^k(C^{\bullet}\otimes_{\mathbb{Z}} K)\simeq\frac{\ker d_k}{\mathsf{Im} \ d_{k-1}}\otimes_{\mathbb{Z}} K\neq 0 ,$$ 
In other words, there is $x\in\ker d_k$ such that $\mathbb{Z}x\bigcap\mathsf{Im} d_{k-1}=0$. Vice versa, if there exists 
$x\in\ker d_k$ such that $\mathbb{Z}x\bigcap\mathsf{Im} d_{k-1}=0$, then $H^k(C^{\bullet}\otimes_{\mathbb{Z}} K)\neq 0$.

Let $X$ be a direct summand of $\ker d_k$ such that $X$ is finitely generated (free)
$\mathbb{Z}$-module and $x\in X$. Observe that if $x\otimes 1\in \mathsf{Im} \ (d_{k-1}\otimes_{\mathbb{Z}} F)$, then
$x\otimes 1\in \mathsf{Im} \ (d_{k-1}\otimes_{\mathbb{Z}} F)\bigcap X\otimes_{\mathbb{Z}} F$ and the last superspace is spanned by the elements $y\otimes 1, y\in \mathsf{Im}\ d_{k-1}\bigcap X$. Choose a basis of $X$, say $x_1, \ldots , x_t$, such that $n_1 x_1, \ldots , n_t x_t$ is a basis of $\mathsf{Im}\ d_{k-1}\bigcap X$ for some non-negative integers $n_1, \ldots , n_t$. Without loss of generality one can assume that $x_1=x$. Then $n_1=0$ and our lemma obviously follows.  
\end{proof}
\begin{rem}\label{compareofdim}
We obtain as a by-product that if $\dim_K H^k(C^{\bullet}\otimes_{\mathbb{Z}} K)<\infty$, then
$\dim_K H^k(C^{\bullet}\otimes_{\mathbb{Z}} K)\leq \dim_F H^k(C^{\bullet}\otimes_{\mathbb{Z}} F)$ for any field $F$. 
\end{rem}
\begin{pr}\label{passagefromcharptozero}
Let $char K=0$. If $H^k(G/B, K^{\epsilon}_{\lambda})\neq 0$ for some $k\geq 0$ and $\lambda$, then 
$H^k(G/B, F^{\epsilon}_{\lambda})\neq 0$ for any field $F$.
\end{pr}
\begin{proof}
By Lemma \ref{indandfixedpoint} $H^k(G/B, F^{\epsilon}_{\lambda})\simeq H^k(B, F^{\epsilon}_{\lambda}\otimes F[G])$.
The Hopf superalgebra $F[B]$ has a $\mathbb{Z}$-form $\mathbb{Z}[B]=\mathbb{Z}[G]/I$, where a Hopf superideal $I$ is generated by the elements $c_{w(i), w(j)}, 1\leq i< j\leq m+n$. Besides, a $B$-supermodule $M=F^{\epsilon}_{\lambda}\otimes F[G]$ has a 
$\mathbb{Z}$-form $M_{\mathbb{Z}}=\mathbb{Z}^{\epsilon}_{\lambda}\otimes\mathbb{Z}[G]$. The cohomology $H^{\bullet}(B, M)$ can be computed using 
{\it Hocshchild complex} $C^{\bullet}(B, M)=\{M\otimes F[B]^{\otimes n}\}$ (cf. \cite{jan, drup}). Moreover, this complex has a $\mathbb{Z}$-form $C^{\bullet}_{\mathbb{Z}}(B, M)=\{M_{\mathbb{Z}}\otimes\mathbb{Z}[B]^n\}$ consisting of free $\mathbb{Z}$-(super)modules. The statement follows by Lemma \ref{complex}. 
\end{proof}
\begin{cor}\label{positivetozero}
If $H^k(G/B, F^{\epsilon}_{\lambda})=0$ for a field $F$, then 
$H^k(G/B, K^{\epsilon}_{\lambda})=0$ for any field $K$ of zero characteristic.
\end{cor}
For the sake of simplicity assume that $B_{ev}=(B_{id})_{ev}$. In the notations of seventh section, $w=w_1, w_0=id$. Thus $\Phi^+_0=\{\epsilon_i-\epsilon_j | 1\leq i < j\leq m, \mbox{or} \ m+1\leq i < j\leq m+n\}$. For any positive integer $s$  
let $\underline{s}$ denote the interval $\{1, \ldots, s\}$. 

Let $\beta_i$ denote the weight $\epsilon_i-\epsilon_{i+1}$, where $i\neq m$. 
The cardinality of $\{\gamma\in\Phi^+_1|
(\gamma, \beta_i^{\vee})=1\}$ is denoted by $k_i$. 
More precisely, if $1\leq i\leq m-1$, then
$$
k_i=(w^{-1}(\underline{m+n}\setminus\underline{m})\setminus\underline{w^{-1}(i)})^{\sharp} +(\underline{w^{-1}(i+1)}\bigcap w^{-1}(\underline{m+n}\setminus\underline{m}))^{\sharp}=$$
$$n+(w^{-1}(\underline{m+n}\setminus\underline{m})\bigcap(\underline{w^{-1}(i+1)}\setminus\underline{w^{-1}(i)}))^{\sharp} .
$$
Finally, if $m+1\leq i\leq m+n-1$, then 
$$
k_i=(\underline{w^{-1}(i)}\bigcap w^{-1}(\underline{m}))^{\sharp}+(w^{-1}(\underline{m})\setminus\underline{w^{-1}(i+1)})^{\sharp}=
$$
$$
m-(w^{-1}(\underline{m})\bigcap(\underline{w^{-1}(i+1)}\setminus\underline{w^{-1}(i)}))^{\sharp} .
$$
\begin{theorem}\label{advancedKempf}
Assume that for each $\beta_i$ the weight $\lambda$ satisfies $(\lambda, \beta_i^{\vee})\geq k_i$. Then
$H^k(G/B, K^{\epsilon}_{\lambda})=0$ for all $k\geq 1$. In particular, $ch(H^0(G/B, K_{\lambda}^{\epsilon}))=ch(\chi(B, \lambda^{\epsilon}))$. 
\end{theorem}
\begin{proof}
By Corollary \ref{positivetozero} it remains to prove our theorem for any field of positive characteristic.
Combining Proposition \ref{filtration} with the standard long exact sequence arguments it remains to show that all
cohomology groups $$H^k(G_{ev}/B_{ev}G_{ev, 1}, ind^{B_{ev}G_{ev, 1}}_{B_{ev}} K_{\lambda-\pi})\simeq H^k(G_{ev}/B_{ev},
K_{\lambda-\pi})$$ are equal to zero.
The condition $(\lambda, \beta^{\vee}_i)\geq k_i$ guarantees that $(\lambda-\pi, \beta^{\vee}_i)\geq 0$ for any $\beta_i$.  
Kempf's vanishing theorem infers the first statement. The second statement follows by Proposition \ref{characterformula}. 
\end{proof}

\section{Cohomologies of $GL(2|1)$}

Let $G=GL(2|1)$. By Grothendieck's vanishing theorem $H^k(G/B, M)=0$ for any $B$-supermodule $M$ whenever $k\geq 2$. 
Thus if $H^0(\lambda^{\epsilon})=0$, then $-ch(\chi(B, \lambda^{\epsilon}))=ch(H^1(\lambda^{\epsilon}))$. Symmetrically,
if $H^1(\lambda^{\epsilon})=0$, then $ch(\chi(B, \lambda^{\epsilon}))=ch(H^0(\lambda^{\epsilon}))$.

There are three representatives of conjugacy classes for Borel supersubgroups in $G$. They are
$B^-_w, w\in\{1, (23), (132)\}$. Assume that $char K=p>0$ and $K$ is perfect. Matrices from $B^-_{(23)}$ have the form
$$\left(\begin{array}{ccc}
* & 0 & 0 \\
* & * & * \\
* & 0 & *
\end{array}\right).
$$
Respectively, matrices from $B^-_{(132)}$ have the form
$$\left(\begin{array}{ccc}
* & 0 & * \\
* & * & * \\
0 & 0 & *
\end{array}\right).
$$
It is clear that they are odd adjuacent via $\epsilon_1 -\epsilon_3$. 

Set $B=B^-_{(23)}$. Then $\alpha_1=\epsilon_1-\epsilon_3, \alpha_2=\epsilon_3-\epsilon_2$ and $\rho=0$.
The only even positive root is $\beta=\beta_1=\alpha_1+\alpha_2$. Besides, $k_1=2$.

By the above convention 
$H^k(G/B, K^{\epsilon}_{\lambda})$ is denoted by $H^k(\lambda^{\epsilon})$. We also denote $H^k(G_{ev}/B_{ev}, K_{\lambda})$ by $H^k_{ev}(\lambda)$.

Since the elements of $U_{12}$ commutes with the elements of $U^+$, $U_{12}$ acts identically on $V=K^{\epsilon}_{\lambda}\otimes K[(U^+)_1]|_{B_{ev}G_{ev, 1}}$. For the elements of $U_{21}$ the following formula hold :
$$\left(\begin{array}{ccc}
1 & c_{12} & c_{13} \\
0 & 1 & 0 \\
0 & c_{32} & 1
\end{array}\right) U_{21}(a)=$$$$\left(\begin{array}{ccc}
1+c_{12}a & 0 & 0 \\
a & \frac{1+(c_{12}+c_{13}c_{23})a}{(1+c_{12}a)^2} & -\frac{c_{13}a}{1+c_{12}a} \\
c_{32}a & 0 & \frac{1+(c_{12}+c_{13}c_{23})a}{1+c_{12}a}
\end{array}\right)\left(\begin{array}{ccc}
1 & \frac{c_{12}}{1+c_{12}a} & \frac{c_{13}}{1+c_{12}a} \\
0 & 1 & 0 \\
0 & \frac{c_{32}}{1+c_{12}a} & 1
\end{array}\right).$$
Thus $U_{21}(a)$ acts on $V$ by the rule :
$$c_{12}^k\mapsto (1+c_{12}a)^{\lambda_1-2\lambda_2-\lambda_3-k}(1+(c_{12}+c_{13}c_{32})a)^{\lambda_2+\lambda_3}c_{12}^k;
$$ 
$$c_{12}^k c_{13}\mapsto (1+c_{12}a)^{\lambda_1-\lambda_2-k-1}c_{12}^k c_{13}; \ c_{12}^k c_{32}\mapsto (1+c_{12}a)^{\lambda_1-\lambda_2-k-1}c_{12}^k c_{32};
$$
$$
c_{12}^k c_{13}c_{32}\mapsto (1+c_{12}a)^{\lambda_1-\lambda_2-k-2}c_{12}^k c_{13}c_{32}.
$$
The above computations immediately show that $K_{\lambda}\otimes K[c_{12}]c_{13}$, $K_{\lambda}\otimes K[c_{12}]c_{32}$ and $K_{\lambda}\otimes K[c_{12}]c_{13}c_{32}$ are $B_{ev}G_{ev, 1}$-submodules of $V$. The first two submodules are isomorphic to
$ind^{B_{ev}G_{ev, 1}}_{B_{ev}} K_{\lambda-\alpha_1}$ and $ind^{B_{ev}G_{ev, 1}}_{B_{ev}} K_{\lambda-\alpha_2}$ respectively.
The third submodule is isomorphic to $ind^{B_{ev}G_{ev, 1}}_{B_{ev}} K_{\lambda-\beta}$. Moreover, let us denote the direct sum of these submodules by $W$. Then $V/W\simeq ind^{B_{ev}G_{ev, 1}}_{B_{ev}} K_{\lambda}$.
\begin{pr}\label{thecase23}
The following statments hold :
\begin{enumerate}
\item If $(\lambda, \beta^{\vee})\geq 1$, then $H^1(\lambda^{\epsilon})=0$ but $H^0(\lambda^{\epsilon})\neq 0$;
\item If $(\lambda, \beta^{\vee})=0$, then $H^1(\lambda^{\epsilon})$ and $H^0(\lambda^{\epsilon})$ are one-dimensional $G$-supermodules (of the same weight $\lambda$) if and only if $\lambda$ has a form $(a, a, -a+pt)$ if and only if $\lambda$ is atypical, otherwise $H^1(\lambda^{\epsilon})=H^0(\lambda^{\epsilon})=0$;
\item If $(\lambda, \beta^{\vee})<0$, then $H^1(\lambda^{\epsilon})\neq 0$ and $H^0(\lambda^{\epsilon})=0$.  
\end{enumerate}
\end{pr}
\begin{proof}
We have the long exact sequence :
$$0\to H^0_{ev}(\lambda-\alpha_1)\oplus H^0_{ev}(\lambda-\alpha_2)\oplus H^0_{ev}(\lambda-\beta)
\to H^0(\lambda^{\epsilon})|_{G_{ev}}\to H^0_{ev}(\lambda)\to$$ 
$$\to H^1_{ev}(\lambda-\alpha_1)\oplus H^1_{ev}(\lambda-\alpha_2)\oplus H^1_{ev}(\lambda-\beta)
\to H^1(\lambda^{\epsilon})|_{G_{ev}}\to H^1_{ev}(\lambda)\to 0.$$ 
Assume that $(\lambda, \beta^{\vee})\geq 2$. By Theorem \ref{advancedKempf}
$H^1({\lambda}^{\epsilon})=0$ but $H^0_{ev}(\lambda)\neq 0$, hence $H^0(\lambda^{\epsilon})\neq 0$. 

Let $(\lambda, \beta^{\vee})=1$. Then 
$$H^1_{ev}(\lambda)=H^1_{ev}(\lambda-\alpha_1)=H^1_{ev}(\lambda-\alpha_2)=
H^0_{ev}(\lambda-\beta)=0
$$
and we have the long exact sequence :
$$0\to H^0_{ev}(\lambda-\alpha_1)\oplus H^0_{ev}(\lambda-\alpha_2)
\to H^0(\lambda^{\epsilon})|_{G_{ev}}\to
H^0_{ev}(\lambda)\to
H^1_{ev}(\lambda-\beta)\to$$
$$\to H^1(\lambda^{\epsilon})|_{G_{ev}}\to 0 .$$ 
By Serre duality $H^1_{ev}(\lambda-\beta)\simeq H^0_{ev}(-\lambda)^* =0$ (cf. \cite{jan}, p.203).
It infers again $H^1(\lambda^{\epsilon})=0$ but $H^0(\lambda^{\epsilon})\neq 0$. 

Now, assume that $(\lambda, \beta^{\vee})=0$. Then
$$H^0_{ev}(\lambda-\alpha_1)=H^0_{ev}(\lambda-\alpha_2)=H^0_{ev}(\lambda-\beta)=0$$
and the long exact sequence is converted into :
$$0\to H^0(\lambda^{\epsilon})|_{G_{ev}}\to H^0_{ev}(\lambda)\to 
H^1_{ev}(\lambda-\alpha_1)\oplus H^1_{ev}(\lambda-\alpha_2)\oplus H^1_{ev}(\lambda-\beta)\to
H^1(\lambda^{\epsilon})|_{G_{ev}}\to 0 .$$
Again, by Serre duality $H^1_{ev}(\lambda-\alpha_1)\simeq H^0_{ev}(-(\lambda+\alpha_2))^*=0,$
$H^1_{ev}(\lambda-\alpha_2)\simeq H^0_{ev}(-(\lambda+\alpha_1))^*=0$ and $H^1_{ev}(\lambda-\beta)\simeq H^0_{ev}(-\lambda)^*$.
So, the long exact sequence has the form :
$$0\to H^0(\lambda^{\epsilon})|_{G_{ev}}\to H^0_{ev}(\lambda)\to H^0_{ev}(-\lambda)^*\to H^1(\lambda^{\epsilon})|_{G_{ev}}\to 0 .$$
By Weyl's character formula $ch(H^0_{ev}(\lambda))=e^{\lambda}=ch(H^0_{ev}(-\lambda)^*)$ (cf. \cite{jan}, II.5.10).
Thus either $H^0(\lambda^{\epsilon})$ and $H^1(\lambda^{\epsilon})$ are one dimensional $G$-supermodules or $H^0(\lambda^{\epsilon})=H^1(\lambda^{\epsilon})=0$.
By Lemma \ref{characters} the first opportunity holds if and only if $\lambda=(a, a, -a+pt)$ for
some integers $a, t$ if and only if $\lambda$ is atypical.  

Finally, if $(\lambda, \beta^{\vee})<0$, then $H^0(\lambda^{\epsilon})=0$ but $H^1(\lambda^{\epsilon})\neq 0$.
In fact, at least $H^1_{ev}(\lambda-\beta)\simeq H^0_{ev}(-\lambda)^*$ is not equal to zero. 
\end{proof}
Now, set $B=B^-_{(132)}$. In the above notations $V=K^{\epsilon}_{\lambda}\otimes K[(U^+)_1]|_{B_{ev}G_{ev, 1}}$ has a filtration
by $B_{ev}G_{ev, 1}$-modules as :
$$W=ind^{B_{ev}G_{ev, 1}}_{B_{ev}} K_{\lambda}\oplus ind^{B_{ev}G_{ev, 1}}_{B_{ev}} K_{\lambda-\alpha_2}\oplus ind^{B_{ev}G_{ev, 1}}_{B_{ev}} K_{\lambda+\alpha_1-\alpha_2}  \subseteq V$$
and $V/W\simeq ind^{B_{ev}G_{ev, 1}}_{B_{ev}} K_{\lambda+\alpha_1}$. 
For example, the isomorphism $K_{\lambda}\otimes K[c_{12}]c_{32}\simeq ind^{B_{ev}G_{ev, 1}}_{B_{ev}} K_{\lambda-\alpha_2}$ 
is given by 
$$c_{12}^k c_{32}\mapsto  
\{\begin{array}{c} 
c_{12}^{k-1}, k\geq 1, \\
c_{12}^{p-1}, k=0.
\end{array}$$
Since $c_{12}^p=0$ and $U_{12}(a)$ takes $c_{12}^k c_{32}$ to $(1+c_{12}a)^{\lambda_1-\lambda_2-k}c_{12}^k c_{32}$, the isomorphism is defined correctly.
\begin{pr}\label{thecase132}
The following statements hold :
\begin{enumerate}
\item If $(\lambda, \beta^{\vee})\geq 0$, then $H^1(\lambda^{\epsilon})=0$ but $H^0(\lambda^{\epsilon})\neq 0$;
\item If $(\lambda, \beta^{\vee})=-1$, then $H^1(\lambda^{\epsilon})=H^0(\lambda^{\epsilon})=0$;
\item If $(\lambda, \beta^{\vee})\leq -2$, then $H^1(\lambda^{\epsilon})\neq 0$ and $H^0(\lambda^{\epsilon})=0$.
\end{enumerate}
\end{pr}
\begin{proof}
The long exact sequence is :
$$0\to H^0_{ev}(\lambda)\oplus H^0_{ev}(\lambda-\alpha_2)\oplus H^0_{ev}(\lambda+\alpha_1-\alpha_2)
\to H^0(\lambda^{\epsilon})|_{G_{ev}}\to H^0_{ev}(\lambda+\alpha_1)\to$$ 
$$\to H^1_{ev}(\lambda)\oplus H^1_{ev}(\lambda-\alpha_2)\oplus H^1_{ev}(\lambda+\alpha_1 -\alpha_2)
\to H^1(\lambda^{\epsilon})|_{G_{ev}}\to H^1_{ev}(\lambda+\alpha_1)\to 0.$$ 
Arguing as above we see that $H^1(\lambda^{\epsilon})=0$ but $H^0(\lambda^{\epsilon})\neq 0$, provided $(\lambda, \beta^{\vee})\geq 0$. If $(\lambda, \beta^{\vee})\leq -2$, then $H^0(\lambda^{\epsilon})=0$ and since
$H^1_{ev}(\lambda-\alpha_2)\simeq H^0(-(\lambda+\alpha_1))^*\neq 0$, $H^1(\lambda^{\epsilon})\neq 0$. 

Finally, let $(\lambda, \beta^{\vee})=-1$. The long exact sequence is converted into :
$$0\to H^0(\lambda^{\epsilon})|_{G_{ev}}\to H^0_{ev}(\lambda+\alpha_1)\to
H^0_{ev}(-(\lambda+\alpha_1))^*\to H^1(\lambda^{\epsilon})|_{G_{ev}}\to 0 .
$$
Thus either 
$H^0(\lambda^{\epsilon})|_{G_{ev}}\simeq H^0_{ev}(\lambda+\alpha_1)$
and $H^1(\lambda^{\epsilon})|_{G_{ev}}\simeq H^0_{ev}(-(\lambda+\alpha_1))^*$ or
$H^0(\lambda^{\epsilon})=H^1(\lambda^{\epsilon})=0$. The first opportunity would imply
that $H^0(\lambda^{\epsilon})$ and $H^1(\lambda^{\epsilon})$ are simple one dimensional $G$-supermodules and hence $(\lambda, \beta^{\vee})=0$. This contradiction infers that $H^0(\lambda^{\epsilon})$ and $H^1(\lambda^{\epsilon})$ are always
trivial whenever $(\lambda, \beta^{\vee})=-1$. 
\end{proof}
\begin{rem}\label{nopassageforatypical}
Comparing the second statements of the above propositions we see that Proposition \ref{passageviaodd} is no longer true
for atypical weights!
\end{rem}
\begin{rem}\label{charzero}
Since $ch(\chi(B, \lambda^{\epsilon}))$ does not depend on $char K$, Corollary \ref{positivetozero} immediately implies that all statements of the above propositions hold in the case $char K=0$. Only the case $B=B^-_{(23)}$ and $\lambda=(a, a, -a)$ needs some extra arguing. But in this case $H^0(\lambda^{\epsilon})$ always contains a simple supermodule that is isomorphic to $KBer(C)^a\subseteq K[G]_r$ (see Lemma \ref{characters}). It remains to refer to Remark \ref{compareofdim}. 
\end{rem}
\begin{rem}
The results of this section show that the statement of Theorem \ref{advancedKempf} still holds even if for some $i$ the integer $(\lambda, \beta_i^{\vee})$ is less than $k_i$. It rises the natural question to find the minimal value of each $(\lambda, \beta_i^{\vee})$ that would infer the statement of Theorem \ref{advancedKempf}. We conjecture that Theorem \ref{advancedKempf} still holds whenever $(\lambda, \beta_i^{\vee})\geq \min\{m, n\}$ for any $i\neq m$.
\end{rem} 
\begin{center}
Acknowledgement.
\end{center}
This work was suppored by Ministry of Science and Education of Russian Federation, grant 14.V37.21.0359 /0859. We thank for this support.

\end{document}